\documentclass[12pt,twoside]{amsart}
\usepackage[cp1251]{inputenc}
\usepackage{amssymb}
\RequirePackage{srcltx}

\usepackage[T2A]{fontenc}
\usepackage[cp1251]{inputenc}
\usepackage[english]{babel}
\usepackage{color}

\advance\textwidth30mm \advance\hoffset-15mm
\advance\textheight40mm \advance\voffset-20mm
\theoremstyle{plain}
\newtheorem{thm}{Theorem}[section]

\newtheorem{lem}{Lemma}[section]

\newtheorem*{prob}{Problem C}
\newtheorem*{probI}{Problem~A}
\newtheorem*{probII}{Problem~B}

\theoremstyle{remark}
\newtheorem{rem}{Remark}[section]

\numberwithin{equation}{section}
\newcommand\<{\langle}
\renewcommand\>{\rangle}

\begin{document}

\title[Generalized Logan's problems]
{
Uncertainty principles for eventually constant sign bandlimited functions
}

\author{D.~V.~Gorbachev}
\address{D.~Gorbachev, Tula State University,
Department of Applied Mathematics and Computer Science,
300012 Tula, Russia}
\email{dvgmail@mail.ru}

\author{V.~I.~Ivanov}
\address{V.~Ivanov, Tula State University,
Department of Applied Mathematics and Computer Science,
300012 Tula, Russia}
\email{ivaleryi@mail.ru}

\author{S.~Yu.~Tikhonov}
\address{S.~Tikhonov,
Centre de Recerca Matem\`atica, Campus de Bellaterra, Edifici C 08193
Bellaterra, Barcelona, Spain; ICREA, Pg. Llu\'is Companys 23, 08010 Barcelona,
Spain, and Universitat Aut\'onoma de Barcelona}
\email{stikhonov@crm.cat}

\date{\today}



\thanks{The work of D.\,V.~Gorbachev and V.\,I.~Ivanov is  supported  by  the  Russian  Science  Foundation  under  grant 18-11-00199 and performed  in Tula State University.
S.\,Yu.~Tikhonov was partially supported by  MTM 2017-87409-P, 2017 SGR 358, and
 the CERCA Programme of the Generalitat de Catalunya.}

\begin{abstract}
%
%
%
We study the uncertainty principles related to the generalized Logan problem in $\mathbb{R}^{d}$.
Our main result provides the complete solution of the following problem:
for a fixed $m\in \mathbb{Z}_{+}$,  find
\[
\sup\{|x|\colon (-1)^{m}f(x)>0\}\cdot
\sup \{|x|\colon x\in \mathrm{supp}\,\widehat{f}\,\}\to \inf,
\]
where the infimum is taken over all nontrivial positive definite
 bandlimited functions such that $\int_{\mathbb{R}^d}|x|^{2k}f(x)\,dx=0$ for
$k=0,\dots,m-1$ if $m\ge 1$.

We also obtain the uncertainty principle for bandlimited functions related to
the recent result by Bourgain, Clozel, and Kahane.



\end{abstract}

\maketitle
\section{Introduction}

\subsection{Logan's problems.}
Logan stated and proved \cite{Lo83a,Lo83b} the following two extremal problems for real-valued positive definite
bandlimited
functions on $\mathbb{R}$. Since such functions are even, we state these problems {for functions on $\mathbb{R}_+=[0,\infty)$}.

\begin{probI}
Find the smallest $\lambda_0>0$ 
 such that $$f(x)\le 0,\quad x>\lambda_0,$$
where $f$ is a positive definite function {of exponential type at most}~$1$ satisfying
\begin{equation}\label{eq1}
f(x)=\int_{0}^{1}\cos xt\,d\mu(t),\quad d\mu\ge 0,\quad f(0)=1.
\end{equation}

\end{probI}

Logan showed that admissible functions are integrable (even if the measure
$d\mu$ is nonnegative in a neighborhood of the origin), $\lambda_0=\pi$, and
the unique extremizer is
\[
f_0(x)=\frac{\cos^2(x/2)}{1-x^2/\pi^2}=
\frac{\pi}{2}\int_{0}^{1}\sin\pi t\cos xt\,dt,
\]
{Note that  $f_0$ satisfies}
  $\int_{\mathbb{R}_{+}}f_0(x)\,dx=0$.

\begin{probII}
Find the smallest $\lambda_1>0$ such that
  $$f(x)\ge 0,\quad x>\lambda_1,$$
where $f$ is a positive definite integrable function satisfying \eqref{eq1} and having mean value~zero.
\end{probII}

It turns out that admissible functions are integrable with respect to the weight $x^2$, and $\lambda_1=3\pi$. Moreover,
the unique extremizer is
\[
f_1(x)=\frac{\cos^2(x/2)}{(1-x^2/\pi^2)(1-x^2/(3\pi)^2)}=
\frac{3\pi}{4}\int_{0}^{1}(\sin\pi t)^3\cos xt\,dt,
\]
This function satisfies  $\int_{\mathbb{R}_{+}}x^{2}f_1(x)\,dx=0$.


We will study the multivariate generalization of Logan's problems for the
Fourier transforms. In more detail, we consider the $m$-parameter problem,
$m\in \mathbb{Z}_{+}=\{0,1,\ldots\}$, so that, for $d=1$, if $m=0,\,1$ we recover Problems~A and B
respectively.

Let $d\in \mathbb{N}$ and $\mathbb{R}^{d}$ is $d$-dimensional Euclidean space
with the scalar product $\<x,y\>=x_{1}y_{1}+\dots+x_{d}y_{d}$ and the norm
$|x|=\<x,x\>^{1/2}$. Let $B_{\tau}^{d}=\{x\in \mathbb{R}^d\colon |x|\le\tau\}$
be the ball of radius $\tau>0$. Let $Q=\mathbb{R}^d$ or $Q=\mathbb{R}_+.$ As
usual, for a positive measure space $(Q,d\rho)$, let $L^{p}(Q,d\rho)$ denote
the space of measurable functions such that
$\|f\|_{p,d\rho}=\bigl(\int_{Q}|f(x)|^p\,d\rho(x)\bigr)^{1/p}<\infty$,
$L^{\infty}(Q)$ be the space of
the essentially bounded measurable functions,
and $C(Q)$ consists of continuous functions on $Q$.
The Fourier transform of $f$ is given by
\[
\widehat{f}\,(y)=
(2\pi)^{-d/2}\int_{\mathbb{R}^d}f(x)e^{-i\<x,y\>}\,dx,\quad y\in
\mathbb{R}^{d}.
\]
A function $f$ defined on $\mathbb{R}^{d}$ is positive-definite if {for each $N$}
\[
\sum_{i,j=1}^Nc_i\overline{c_j}\,f(x_i-x_j)\ge 0,\quad \forall\,c_1,\dots,c_N\in
\mathbb{C},\quad \forall\,x_1,\dots,x_N\in \mathbb{R}^{d}.
\]
Recall that for a continuous function $f$, by Bochner's theorem, $f$ is positive definite if and only~if
\begin{equation}\label{mu-repr}
f(x)=\int_{\mathbb{R}^d}e^{i\<x,y\>}\,d\mu(x),
\end{equation}
where $\mu$ is a finite positive Borel measure (see, e.g.,
\cite[9.2.8]{Ed79}). In particular, if $f\in L^{1}(\mathbb{R}^{d})$, then
$d\mu(x)=(2\pi)^{-d/2}\widehat{f}\,(x)\,dx$ and $\widehat{f}\ge 0$.

 {In this paper we deal with continuous even functions $f\colon \mathbb{R}^{d}\to \mathbb{R}$, which are
constant sign outside of a ball $B_{\lambda}^{d}$. Denote by $\lambda(f)$ the smallest radius of a ball such that
$f$ is non-positive outside of this ball,
that is,}
\[
\lambda(f)=\sup\{|x|\colon f(x)>0\}.
\]
Thus, functions with
$\lambda(-f)<\infty$ are eventually nonnegative.

A function $f$ is bandlimited if the distributional Fourier transform
$\widehat{f}\,$ (or the measure~$\mu$ in (\ref{mu-repr})) has a compact support. 
Let
\[
\tau(f)=\sup \{|x|\colon x\in \mathrm{supp}\,\widehat{f}\,\}.
\]
By the Paley--Wiener--Schwarz theorem, bandlimited functions $f$ are
restrictions of complex-valued entire functions of spherical exponential type
$\tau(f)$ to $\mathbb{R}^{d}$ (see, e.g.,~\cite{NW78}).


As in the original Logan's problems we are interested in the smallest value of $\lambda(\pm f)$ for continuous positive definite functions
 $f$ with finite type $\tau(f)$. We also assume that the following orthogonality condition holds:
\[
\int_{\mathbb{R}^d}|x|^{2k}f(x)\,dx=0,\quad k=0,\dots,k_{m},\quad f\in
L^{1}(\mathbb{R}^{d},|x|^{2k_{m}}\,dx),
\]
for some integer $k_{m}$, 
  cf.\ the condition in Problem~B.
This condition is equivalent~to
\[
\Delta^{k}\widehat{f}\,(0)=0,\quad k=0,\dots,k_{m},
\]
where $\Delta$ is the Laplace operator, $\Delta^{0}=\mathrm{Id}$.

{One of the main goals} of 
  this paper is to solve the following
\begin{prob}
For $d\in \mathbb{N}$ and $m\in \mathbb{Z}_{+}$, find
\[
\inf \lambda((-1)^{m}f)\tau(f),
\]
where 
the infimum is taken over all nontrivial {continuous positive definite bandlimited functions on}
$\mathbb{R}^{d}$ {such that additionally if $m\ge 1$,
$f\in L^{1}(\mathbb{R}^{d},|x|^{2m-2}\,dx)$ and
$\Delta^{k}\widehat{f}\,(0)=0$, $k=0,\dots,m-1$.}
\end{prob}

{It is worth mentioning that  admissible functions  in problem C
as well as the expression
$\lambda(\pm f)\tau(f)$ are invariant with respect to the dilation
$f_{a}(x)=f(ax)$, $a>0$, since $\lambda(\pm f_{a})=a^{-1}\lambda(\pm f)$ and
$\tau(f_{a})=a\tau(f)$. Note that in Problems A and B we have $\tau(f)=1$.}


Problem C has various  applications, in particular, to investigate Odlyzko's question on zeros of the Dedekind zeta function (see \cite{Lo83b} and  \cite[Sec. 4]{BCK10}). For
 $m={0}$ 
it plays an important role
in several extremal problems in approximation theory (see, e.g.,
\cite{Be99, Go00}).

To formulate our main result, for $\alpha\ge -1/2$ we introduce the even entire function of exponential type~2
\begin{equation}\label{def-fam}
f_{\alpha,m}(t)=\frac{j_{\alpha}^2(t)}{(1-t^2/q_{\alpha,1}^2)\cdots
(1-t^2/q_{\alpha,m+1}^2)},\quad t\in \mathbb{R}_{+},
\end{equation}
where $j_{\alpha}(t)=\Gamma(\alpha+1)(2/t)^{\alpha}J_{\alpha}(t)$ is the
normalized Bessel function and $q_{\alpha,1}<q_{\alpha,2}<\cdots$ are positive
zeros of $J_{\alpha}$.

\begin{thm}\label{thm-F}
For $d\in \mathbb{N}$ and $m\in \mathbb{Z}_{+}$, we have
\[
\inf \lambda((-1)^{m}f)\tau(f)=2q_{d/2-1,m+1},
\]
where the infimum is taken over all admissible  functions in Problem C.
The  function $f_{d/2-1,m}(|x|)$ is
  the unique extremizer up to a positive constant.
{Moreover,
this function satisfies $\Delta^{m}\widehat{f}(0)=0$.}
\end{thm}


We note that the same statement is valid not only for positive definite
functions but also for even functions with nonnegative Fourier transforms in a
neighborhood of the origin. The positive definiteness of $f_{d/2-1,m}$ for
$m=0,\,1$ was established by Yudin \cite{Yu81,Yu97}. In the case $m=0,\,1$
Theorem~\ref{thm-F} was proved in \cite{Go00}.
We prove Theorem~\ref{thm-F} by solving a more general problem for the Dunkl
transform~$\mathcal{F}_{k}$ (see Section~\ref{sec-D}). {In its turn, the
corresponding problem for the Dunkl transform can be reduced to the one-dimensional problem
for the Hankel transform $\mathcal{H}_{\alpha}$, $\alpha\ge -1/2$, in
$(\mathbb{R}_{+},\lambda^{2\alpha+1}\,d\lambda)$.}

The key step in the proof of Theorem \ref{thm-F} is to show the positive
definiteness of $f_{d/2-1,m}$ for $m\ge 2$. Note that since the normalized
Bessel function $j_{d/2-1}(|x|)$ is positive definite it is enough to verify that
$g_{d/2-1,m}(|x|)$ is positive definite, where
\begin{equation}\label{def-gam}
g_{\alpha,m}(t)=\frac{j_{\alpha}(t)}{(1-t^2/q_{\alpha,1}^2)\cdots(1-t^2/q_{\alpha,m+1}^2)}.
\end{equation}
This remarkable fact has been recently established by
Cohn and de~Courcy-Ireland \cite[Proposition~3.1]{CC18}. The method of the
proof is based on the Mehler--Heine formula on interrelation between the Bessel
functions and Gegenbauer polynomials as well as the important result from the
paper \cite{CK07} stating that the polynomial
\[
\frac{P^{(\alpha,\alpha)}_n (z)}{
(z - r_{1,n}) \cdots (z - r_{k,n})}
\]
is a linear combination of $P^{(\alpha,\alpha)}_0 (z),\dots,
P^{(\alpha,\alpha)}_{n-k} (z)$ with nonnegative coefficients for each $k \le
n$, where $r_{1,n}>r_{2,n}>\cdots>r_{n,n}$ are zeros of the Jacobi polynomial
$P^{(\alpha,\alpha)}_n (z)$
 (in the case $k=1,2$, this was proved in \cite{GI00}).
  Cohn and de~Courcy-Ireland used the function $f_{d/2-1,m}$ to obtain lower bounds for energy in the
Gaussian core model (see \cite[Sect.~6]{CC18}).

To solve Logan's problem for the Hankel transform
$\mathcal{H}_{\alpha}$, one should show that
 $g_{\alpha,m}$ is positive definite with respect to $\mathcal{H}_{\alpha}$ for any $\alpha\ge -1/2$ and $m\ge 0$. For
$\alpha=-1/2$, $m=0,\,1$, we arrive at the cosine Fourier transform considered by Logan.
 We will give two proofs of positive definiteness of the function $g_{\alpha,m}$.
 The first one is the direct proof using the Sturm theorem on number on zeros of linear combinations of eigenfunctions
 (see Section~\ref{sec-cheb}). In particular, following this approach, one can obtain the monotonicity
 of the Hankel transform of the function $g_{\alpha,m}$ {on $[0,1]$}.
The second proof extends the one by Cohn and de~Courcy-Ireland for the case of
any $\alpha$ (not necessarily
 half-integer) and is given in
Section~\ref{sec-lim}.

\begin{rem}\label{rem-theta}
Note that the functions  $g_{d/2-1+\theta,m}(|x|)$ and $f_{d/2-1+\theta,m}(|x|)$ are positive definite on $\mathbb{R}^{d}$ for any $\theta\ge 0$
and $m\in \mathbb{Z}_{+}$.
This follows from  \eqref{s-H} below
and the fact that  for any
 $\alpha\ge -1/2$ and $m\in
\mathbb{Z}_{+}$,  
$g_{\alpha,m}$ and $f_{\alpha,m}$ are
positive definite with respect to
  Hankel transform.
This result answers the question by M. Buhmann and  is related to
the theory of  radial basis functions (see, e.g., \cite{Bu03}).

\end{rem}

\subsection{Uncertainty principle relations}
Recently, Bourgain, Clozel, and Kahane \cite{BCK10} have studied the following uncertainty principle problem: find
\[
A_{d}^{+}=\frac{1}{2\pi}\,\inf \lambda(-f)\lambda(-\widehat{f}\,),
\]
where infimum is taken over all even real-valued (nontrivial) functions~$f$
such that $f,\,\widehat{f}\in C(\mathbb{R}^{d})\cap L^{1}(\mathbb{R}^{d})$ and
$f(0)\le 0$, $\widehat{f}\,(0)\le 0$.
They established
\begin{equation}\label{Bd}
\frac{d}{2\pi e}<A_{d}^+<\frac{d+2}{2\pi},\quad d\in \mathbb{N}.
\end{equation}
For further results, see \cite{CG18,GOS17}.
 Cohn and Gon\c{c}alves  \cite{CG18}   proved that 
  $$A^+_{12}=2.$$ Moreover, the authors considered
the following problem: $A_{d}^{-}=(2\pi)^{-1}\inf
\lambda(-f)\lambda(\widehat{f}\,)$ for $f(0)\ge 0$, $\widehat{f}\,(0)\le 0$ and found
$$A^-_{1}=1,\quad A^-_{8}=2,\quad A^-_{24}=4.$$
This question  is closely related to
 the linear programming bound for
the sphere packing problem, which
has been recently  solved in dimensions 
$8$ and $24$ \cite{CKMRV17,Vi17}.


In \cite[Theorem~1.4]{GOS17}, it was shown that an extremizer in the problem
$A^\pm_{d}$ exists and it is a radial function such that
$(2\pi)^{d/2}\widehat{f}(2\pi x)=\pm f(x)$ and $f(0)=0$. In particular, this
implies that the support of $\widehat{f}$ is not compact.

We study problems similar to  that of finding $A_{d}^{\pm}$ for bandlimited functions
and obtain
the following uncertainty principle.


\begin{thm}\label{thm-FU}
Let $d\in \mathbb{N}$, $m,s\in \mathbb{Z}_{+}$.
We have
\[
\inf \lambda((-1)^{m}f)\tau(f)=2q_{d/2+s,m+1},
\]
where the infimum is taken over {all nontrivial even continuous bandlimited functions
$f\in L^{1}(\mathbb{R}^{d},|x|^{2m}\,dx)$ such that}
\[
\begin{cases}
\Delta^{k}\widehat{f}\,(0)=0,&k=0,\dots,m-1,\\
\Delta^{l}f(0)=0,&l=0,\dots,s-1,
\end{cases}
\]
(for  $m=0$ or $s=0$ the corresponding conditions are not assumed) 
and
\[
\Delta^{m}\widehat{f}\,(0)\ge 0,\quad \Delta^{s}f(0)\le 0.
\]
{Each extremizer $f(x)$ has} the form $\,r(x)f_{d/2+s,m}(|x|)$, where
\begin{equation}\label{r-repr}
r(x)=\sum_{j=0}^{s+1}|x|^{2s+2-2j}h_{2j}(x)\ge 0,\quad |x|\ge q_{d/2+s,m+1},
\end{equation}
and $h_{2j}(x)$ are even harmonic polynomials of order at most $2j$ {such that
$h_{0}>0$}, $h_{2j}(0)=0$, $j=1,\dots,s+1$.
{Moreover, $\Delta^{m}\widehat{f}\,(0)=\Delta^{s}f(0)=0$.}
\end{thm}

\begin{rem}\label{rem-FU}
(1) We also obtain the following result  (see Theorem~\ref{thm-D}~(iii)):
\begin{equation}\label{(1)}
\inf \lambda((-1)^{m}f)\tau(f)=2q_{d/2+s-1,m+1},
\end{equation}
where the infimum is taken over all {nontrivial even continuous bandlimited functions}
$f\in L^{1}(\mathbb{R}^{d},|x|^{2m+2s}\,dx)$ such that
\begin{equation}\label{Bd-}
\Delta^{k}\widehat{f}\,(0)=0,\quad k=s,\dots,m+s-1,\quad
\Delta^{m+s}\widehat{f}\,(0)\ge 0.
\end{equation}
The function $f_{d/2+s-1,m+1}(|x|)$ is the unique (up to a positive constant) extremizer. {Moreover, this function satisfies} 
 $\Delta^{m+s}\widehat{f}(0)=0$. 

\smallbreak
(2) For $s=0$ all admissible functions in problem C satisfy   condition (\ref{Bd-}).
  Moreover, the
positive definite function
   $f_{d/2-1,m+1}(|x|)$ is  the unique extremizer in both {problems~C and \eqref{(1)}}.

\smallbreak
(3) If the polynomial $r(x)$ given by \eqref{r-repr} is nonnegative on
$\mathbb{R}^d$, then it is an even homogeneous polynomial of order $2s+2$.
\end{rem}

\begin{rem}\label{rem-FU-}

Let us compare  problems $A^\pm_d$ and Theorem \ref{thm-FU} with $m=s=0$.
From the observations above we note that
  $A_{d}^{\pm}=(2\pi)^{-1}\inf \lambda(f)\lambda(\pm\widehat{f}\,)$ with
$f(0)=\widehat{f}(0)=0$. For   bandlimited $f$, we have
$\lambda(\pm\widehat{f}\,)\le \tau(f)$ and therefore,  $A_{d}^{\pm}\le (2\pi)^{-1}\inf
\lambda(f)\tau(f)$. 
In particular, we get $A_{d}^{\pm}\le \pi^{-1}q_{d/2,1}$ for any $d\in \mathbb{N}$.
If  $d=1$ we arrive at the sharp bound $A_{d}^{\pm}\le 1$. If $d\to \infty$, we derive
\[
A_{d}^{\pm}\le \frac{d}{2\pi}\,(1+o(1)).
\]
The latter corresponds to \eqref{Bd} but it is less interesting since
$q_{\alpha,1}=\alpha+c\alpha^{1/3}+O(\alpha^{-1/3})$, where $c=1.855\cdots$ \cite[Sec.~7.9]{BE53}.

\end{rem}

\begin{rem}
It is also worth mentioning the   related results in metric geometry.
Let $L\subset \mathbb{R}^{d}$ be a lattice of rank $d$, $\lambda_{1}(L)$
be the first successive minimum of $L$, $\mu(L)$ be the covering radius of $L$,
and $L^{*}$ be a dual lattice. One of the important problems is to find the infimum of
 $\mu(L)\lambda_{1}(L^{*})$. There exists a self-dual lattice $L_{d}$ such that \cite{Ba93}
\[
\frac{d}{2\pi e}\,(1+o(1))\le \mu(L_{d})\lambda_{1}(L_{d}^{*})\quad
 \text{as}\quad d\to \infty.
\]
Yudin showed \cite{Yu96} that $\mu(L)\lambda_{1}(L^{*})\le
(2\pi)^{-1}
 \lambda(f)\tau(f)
$
 for any admissible  function in Problem C with $m=0$.
 This and Theorem \ref{thm-F} imply
\[
\frac{d}{2\pi e}\,(1+o(1))\le \mu(L)\lambda_{1}(L^{*})\le
\frac{d}{2\pi}\,(1+o(1)),
\]
cf.\ \eqref{Bd} (see also \cite{Ba93}). 

\end{rem}

\subsection{Structure of the paper}
Section~\ref{sec-H} contains some auxiliary results on the Hankel
transform~$\mathcal{H}_{\alpha}$ as well as the Gauss- and Radau-type quadrature
formulas with zeros of Bessel functions as nodes.

In Section~\ref{sec-LH}, we
give the solution of the generalized Logan problem for Hankel transform
(see Theorem~\ref{thm-H}). 
{Section~\ref{uncer} provides the uncertainty principle relations for
bandlimited functions in $(\mathbb{R}_{+},t^{2\alpha+1}\,dt)$ (see
Theorem~\ref{thm-HUG}).}

In Section~\ref{sec-HU3}, we study
the problem of finding
 the 
 {smallest}
 interval
 {containing at least~$n$} zeros of functions represented by
$
f(\lambda)=\int_{0}^{1}j_{\alpha}(\lambda t)\,d\sigma(t) 
$ 
 with a nonnegative bounded
Stieltjes measure
 $d\sigma$. We will see that extremizers in this problem and Problem C are closely related (Remark \ref{newrem}).


In Section~\ref{sec-D}, we solve the multidimensional Logan problem for the
Dunkl transform 
(see Theorem~\ref{thm-D}) reducing this problem to the corresponding
 problems for the Hankel transforms (Theorems \ref{thm-H} and
\ref{thm-HUG}).
Theorems~\ref{thm-F} and \ref{thm-FU} dealing with for the Fourier
transform  are partial cases of Theorem~\ref{thm-D}.


In Section~\ref{sec-cheb}, we prove that
the normalized Bessel functions form the Chebyshev system.
Section~\ref{sec-lim} contains the proof of positive definiteness of the
function $g_{\alpha,m}$ based on the Mehler--Heine formula for Jacobi
polynomials.

\section{Notation and
auxiliary results
}\label{sec-H}

Useful facts on harmonic analysis involving Hankel transform
$\mathcal{H}_{\alpha}$ 
in $(\mathbb{R}_{+},t^{2\alpha+1}\,dt)$, $\alpha\ge-1/2$,
can be founded in
\cite{BE53,GI15,LS70}. For the reader's convenience we recall some of~them.

Let
\begin{equation}\label{B-def}
B_{\alpha}=\frac{1}{t^{2\alpha+1}}\Bigl(\frac{d}{dt}\,t^{2\alpha+1}\,\frac{d}{dt}\Bigr)=
\frac{d^{2}}{dt^{2}}+\frac{2\alpha+1}{t}\,\frac{d}{dt},
\end{equation}
be the Bessel differential operator. 
{The normalized Bessel function
$j_{\alpha}(z)$ satisfies $B_{\alpha}j_{\alpha}(\lambda t)=-\lambda^{2}j_{\alpha}(\lambda t)$ and  is
given by

\begin{equation}\label{eq3}
j_{\alpha}(z)=2^{\alpha}\Gamma(\alpha+1)\,
\frac{J_{\alpha}(z)}{z^{\alpha}}=
\sum_{k=0}^{\infty}\frac{(-1)^{k}
\Gamma(\alpha+1)(z/2)^{2k}}{k!\,\Gamma(k+\alpha+1)},
\end{equation}
where
 $J_{\alpha}(z)$ is the Bessel function of order $\alpha$.
In particular,
$j_{-1/2}(z)=\cos z$ and $j_{1/2}(z)=z^{-1}\sin z.$
Moreover, the normalized Bessel function is the even entire function of exponential type $1$, satisfying
$j_{\alpha}(z)=
\prod_{k=1}^{\infty}\bigl(1-\frac{z^{2}}{q_{\alpha,k}^{2}}\bigr)$,
where $q_{\alpha,1}<q_{\alpha,2}<\dots$ are positive zeros of $J_{\alpha}$.

The known formulas for Bessel functions imply
\begin{equation}\label{j-diff-rec}
\frac{d}{dz}\,j_{\alpha}(z)=-\frac{z}{2(\alpha+1)}\,j_{\alpha+1}(z)=
\frac{2\alpha}{z}\,(j_{\alpha-1}(z)-j_{\alpha}(z)),
\end{equation}
\begin{equation}\label{j1-diff}
\frac{d}{dz}\,(z^{2\alpha+2}j_{\alpha+1}(\lambda z))=
2(\alpha+1)z^{2\alpha+1}j_{\alpha}(\lambda z),
\end{equation}
and
\begin{equation}\label{j-int}
\int_{0}^{z}j_{\alpha}(at)j_{\alpha}(bt)t^{2\alpha+1}\,dt=
\frac{z^{2\alpha+2}\{a^{2}j_{\alpha+1}(az)j_{\alpha}(bz)-
b^{2}j_{\alpha}(az)j_{\alpha+1}(bz)\}}{2(\alpha+1)(a^{2}-b^{2})}.
\end{equation}

For {$\lambda\in \mathbb{R}$}, we have
\begin{equation}\label{j-bound}
|j_{\alpha}(\lambda)|\le j_{\alpha}(0)=1,
\end{equation}
and for $|z|\to \infty$, $\mathrm{Re}\,z\ge 0$,
\begin{equation}\label{eq6}
z^{\alpha+1/2}j_{\alpha}(z)=\frac{2^{\alpha+1/2}\Gamma(\alpha+1)}{\Gamma(1/2)}\Bigl(\cos
\Bigl(z-\frac{\pi(2\alpha+1)}{4}\Bigr)+O(|z|^{-1}e^{|\mathrm{Im}\,z|})\Bigr).
\end{equation}

For $\alpha>-1/2$,
we also have Poisson's integral representation
\begin{equation}\label{eq7}
j_{\alpha}(\lambda)=c_{\alpha}\int_{0}^{1}
\left(1-t^{2}\right)^{\alpha-1/2}\cos{}(\lambda t)\,dt,\quad
c_{\alpha}=\frac{\Gamma(\alpha+1)}{\Gamma(1/2)\Gamma(\alpha+1/2)}.
\end{equation}
Then using
\[
(-1)^{m}\Bigl(\cos \lambda- \sum_{k=0}^{m-1}\frac{(-1)^{k}
\lambda^{2k}}{(2k)!}\Bigr)\ge 0,\quad m\in\mathbb{N},\quad \lambda\ge 0,
\]
Poisson's representation gives
\begin{equation}\label{eq8}
\psi_m(\lambda)=(-1)^{m}\Bigl(j_{\alpha}(\lambda)-\sum_{k=0}^{m-1}\frac{(-1)^{k}
\Gamma(\alpha+1)(\lambda/2)^{2k}}{k!\,\Gamma(k+\alpha+1)}\Bigr)\ge 0.
\end{equation}

Define
\begin{equation}\label{def-b-nu}
d\nu_{\alpha}(t)=b_{\alpha}t^{2\alpha+1}\,dt,\quad t\in \mathbb{R}_{+},\quad
b_{\alpha}^{-1}=
2^{\alpha}\Gamma(\alpha+1).
\end{equation}
The Hankel transform is given by
\[
\mathcal{H}_{\alpha}(f)(\lambda)=\int_{0}^{\infty} f(t)j_{\alpha}(\lambda
t)\,d\nu_{\alpha}(t),\quad \lambda\in \mathbb{R}_{+}.
\]
It is an unitary operator in $L^{2}(\mathbb{R}_+,d\nu_{\alpha})$ and
$\mathcal{H}_{\alpha}^{-1}=\mathcal{H}_{\alpha}$.



{If $f\in L^{1}(\mathbb{R}_+, d\nu_{\alpha})\cap C(\mathbb{R}_+)$ and
$\mathcal{H}_{\alpha}(f)\in L^{1}(\mathbb{R}_+, d\nu_{\alpha})$, then, for any}
$t\in\mathbb{R}_+$, one has the inversion formula 
\begin{equation}\label{inv-H}
f(t)=\int_{0}^{\infty}\mathcal{H}_{\alpha}(f)(\lambda)j_{\alpha}(\lambda
t)\,d\nu_{\alpha}(\lambda).
\end{equation}
We also recall the homogeneity property
$\mathcal{H}_{\alpha}(f_{a})(\lambda)=a^{-2\alpha-2}\mathcal{H}_{\alpha}(f)(\lambda/a)$, where
 $f_{a}(t)=f(at)$, $a>0$. Note that
the Hankel transform is a particular case of the one-dimensional Dunkl transform associated with the reflection group $\mathbb{Z}_{2}$ \cite{Ro03}, see Section~\ref{sec-D}.

Let 
$\mathcal{B}_{\alpha}^{\tau}(\mathbb{R}_{+})$
{be the 
class of even} entire functions $f$ of exponential type at most $\tau>0$
such that
the restriction of $f$ to $\mathbb{R}_+$
belongs to $L^{1}(\mathbb{R}_+, d\nu_{\alpha})$. For such functions {one has $|f(z)|\le
\|f\|_{C(\mathbb{R}_{+})}e^{\tau|\mathrm{Im}\,z|}$}, $\forall\,z\in \mathbb{C}$.
Furthermore, the Paley--Wiener theorem states that 
$f\in 
\mathcal{B}_{\alpha}^{\tau}(\mathbb{R}_{+})$ if and only if
$f\in L^{1}(\mathbb{R}_+, d\nu_{\alpha})\cap C(\mathbb{R}_+)$ and
$\mathrm{supp}\,\mathcal{H}_{\alpha}(f)\subset [0,\tau]$
(see \cite[Sect.~5]{Ko75}, \cite[Sect.~5]{AJ05}, and \cite{GI19}).



The following result (\cite{GR95,GhaFra98}, see also \cite{GI15}) provides the Gauss  and {Radau (with multiple nodes)} quadrature formulas for $\mathcal{B}_{\alpha}^{\tau}(\mathbb{R}_{+})$ 
functions.

\begin{lem}
For any function $f\in \mathcal{B}_{\alpha}^{\tau}(\mathbb{R}_{+})$ 
one has
\begin{align}
\Bigl(\frac{\tau}{2}\Bigr)^{2\alpha+2}\int_{0}^{\infty}f(\lambda)\,d\nu_{\alpha}(\lambda)&=
\sum_{k=1}^{\infty}\gamma_{k}f\Bigl(\frac{2q_{\alpha,k}}{\tau}\Bigr)
\label{q-G}\\ &=\sum_{l=0}^{r-1}\alpha_{l,r}f^{(2l)}(0)+
\sum_{k=1}^{\infty}\gamma_{k,r}f\Bigl(\frac{2q_{\alpha+r,k}}{\tau}\Bigr),\quad
r\in \mathbb{N}. \label{q-M}
\end{align}
The series in \eqref{q-G} and \eqref{q-M} converge absolutely and the weights
$\gamma_k$, $\gamma_{k,r}$, $\alpha_{r-1,r}$ are positive.
\end{lem}

\begin{rem}
(1) Formula \eqref{q-M}  was formulated in \cite{GhaFra98} under the more restrictive condition
$f(\lambda)=O(\lambda^{-\delta})$, $\lambda\to +\infty$, $\delta>2\alpha+2$.
However,
\eqref{q-G} was obtained  for any
$f\in L^{1}(\mathbb{R}_+, d\nu_{\alpha})$ \cite{GR95,GI15}. It is easy to see that
 \eqref{q-M} follows from  \eqref{q-G}.
 Indeed, assuming 
  $\tau=2$, one applies
\eqref{q-G} with $d\nu_{\alpha+r}$, $r\ge 1$ to the function
\[
g(\lambda)=\lambda^{-2r}\Bigl(f(\lambda)-j_{\alpha+r}^{2}(\lambda)
\sum_{j=0}^{r-1}(fj_{\alpha+r}^{-2})^{(2j)}(0)\,\frac{\lambda^{2j}}{(2j)!}\Bigr)\in
\mathcal{B}_{\alpha+r}^{2}(\mathbb{R}_{+}).
\]
Straightforward calculations give \eqref{q-M}.

(2) One has
$\alpha_{r-1,r}=c_{\alpha,r}\int_{0}^{\infty}j_{\alpha+r}^{2}(\lambda)\,d\nu_{\alpha+r-1}(\lambda)>0$
{with some $c_{\alpha,r}>0$,}
see \cite{GhaFra98}.
\end{rem}

To construct extremizers for Problem C, we will use the
generalized translation ope\-rator
$T_{\alpha}^t$ 
given by, for $x,t\in\mathbb{R}_+$, 
\begin{equation}\label{eq10}
T_{\alpha}^tf(x)=
\begin{cases}
\frac{1}{2}\,(f(x+t)+f(|x-t|),& \alpha=-1/2,\\
c_{\alpha}\int_{0}^{\pi}
f(\sqrt{x^2+t^2-2xt\cos\theta})\sin^{2\alpha}\theta\,d\theta,&
\alpha>-1/2,
\end{cases}
\end{equation}
where $c_{\alpha}$ is from \eqref{eq7} (see, e.g., \cite{Le51,GIT18}).
The translation operator is positive self-adjoint operator,
$T_{\alpha}^tf(x)\in C({\mathbb{R}_+\times \mathbb{R}_+})$ whenever $f\in
C(\mathbb{R}_+)$, and $T_{\alpha}^t$ extends to the space $L^{p}(\mathbb{R}_+,
d\nu_{\alpha})$, $1\le p\le\infty$. 
{It is known that} $T_{\alpha}^{t}j_{\alpha}(\lambda x)=j_{\alpha}(\lambda t)j_{\alpha}(\lambda
x)$, {which implies}
\begin{equation}\label{eq11}
\mathcal{H}_{\alpha}(T_{\alpha}^tf)(\lambda)=
j_{\alpha}(t\lambda)\mathcal{H}_{\alpha}(f)(\lambda).
\end{equation}
Moreover,
$\mathrm{supp}\,T_{\alpha}^tf(x)\subset [0,a+t]$ if
$\mathrm{supp}\,f\subset [0,a]$.

By means of the operator $T_{\alpha}^t$ we define the positive convolution operator
\[
(f_{1}\ast_{\alpha}
f_{2})(x)=\int_{0}^{\infty}T_{\alpha}^tf_{1}(x)f_{2}(t)\,d\nu_{\alpha}(t),
\]
which satisfies $\mathcal{H}_{\alpha}(f_{1}\ast_{\alpha}
f_{2})=\mathcal{H}_{\alpha}(f_{1})\mathcal{H}_{\alpha}(f_{2})$ and
$\mathrm{supp}\,(f_{1}\ast_{\alpha} f_{2})\subset[0,a_{1}+a_{2}]$ if
$\mathrm{supp}\,f_{i}\subset[0,a_{i}]$.


%


Following Levitan \cite[\S\:11]{Le51}, an even function is called positive definite with respect to
the  Hankel transform $\mathcal{H}_{\alpha}$ {if for each $N$} 
\[
\sum_{i,j=1}^Nc_i\overline{c_j}\,T_{\alpha}^{x_i}f(x_j)\ge 0,\quad
\forall\,c_1,\dots,c_N\in\mathbb{C},\quad
\forall\,x_1,\dots,x_N\in\mathbb{R}_+,
\]
or, equivalently, the matrix $(T_{\alpha}^{x_i}f(x_j))_{i,j=1}^{N}$ is {positive semidefinite}.
By Bochner-type theorem \cite[Theorem~12.1]{Le51}, 
the condition that a continuous function $f$ is positive definite
is equivalent to the fact that $f$ is the Hankel transform of a measure $\sigma$, 
\[
f(\lambda)=\int_{0}^{\infty}j_{\alpha}(\lambda t)\,d\sigma(t),
\]
where $\sigma$ is a non-decreasing function of bounded variation. In
particular, if $f\in L^{1}(\mathbb{R}_+, d\nu_{\alpha})$, then
$d\sigma=\mathcal{H}_{\alpha}(f)\,d\nu_{\alpha}$ and
$\mathcal{H}_{\alpha}(f)\ge 0$.

Moreover, it is easy to see that
if $f$ is positive definite with respect to $\mathcal{H}_{\beta}$, then it is the same with
respect to $\mathcal{H}_{\alpha}$ {for $\alpha<\beta$}, since
\begin{equation}\label{s-H}
\mathcal{H}_{\alpha}(f)(t)=\frac{1}{2^{\beta-\alpha-1}\Gamma(\beta-\alpha)}
\int_{t}^{\infty}s(s^{2}-t^{2})^{\beta-\alpha-1}\mathcal{H}_{\beta}(f)(s)\,ds,\quad
t\in \mathbb{R}_{+}.
\end{equation}
The latter  follows from
 Sonine’s first integral for the Bessel functions:
\begin{equation}\label{s-int}
j_{\beta}(\lambda)=\frac{b_{\beta}^{-1}}{2^{\beta-\alpha-1}\Gamma(\beta-\alpha)}
\int_0^1(1-t^{2})^{\beta-\alpha-1}j_{\alpha}(\lambda t)\,d\nu_{\alpha}(t),
\end{equation}
{where $b_{\beta}$ is defined in \eqref{def-b-nu}.}

Special attention will be paid below to the positive definite functions
$j_{\alpha+1}(\lambda)$ and $j_{\alpha+1}^{2}(\lambda)$. By \eqref{s-int}, we
have
\[
j_{\alpha+1}(\lambda)=b_{\alpha+1}^{-1}\int_0^1j_{\alpha}(\lambda
t)\,d\nu_{\alpha}(t)=b_{\alpha+1}^{-1}\mathcal{H}_{\alpha}(\chi_{[0,1]})(\lambda),
\]
where
$\chi_{I}(t)$ {is the characteristic function of an interval $I$.} 
{Thus,} 
\begin{equation}\label{j2-conv}
j_{\alpha+1}^{2}(\lambda)=
b_{\alpha+1}^{-2}\mathcal{H}_{\alpha}(\chi_{[0,1]}\ast_{\alpha}\chi_{[0,1]})(\lambda)
\end{equation}
and $\mathrm{supp}\,\mathcal{H}_{\alpha}(j_{\alpha+1}^{2})\subset [0,2]$.

We will also use the following two lemmas.

\begin{lem}[\cite{GIV14}]\label{lem-1}
Let $\alpha\ge -1/2$. There exists an even entire function
$\omega_{\alpha}(z)$, $z=x+iy$, of exponential type $2$, positive for $x>0$,
and such that
\[
\omega_{\alpha}(x)\asymp x^{2\alpha+1},\quad x\to +\infty,\quad
|\omega_{\alpha}(iy)|\asymp y^{2\alpha+1}e^{2y},\quad y\to +\infty,
\]
where $F_1\asymp F_2$ means that ${C}^{-1} F_1\le F_2\le C F_1$, $C>0$.
{One can take $\omega_{\alpha}(z)=z^{2m+2}j_{\nu}(z+i)j_{\nu}(z-i)$,
where $\alpha=m-\nu$, $m\in \mathbb{Z}_{+}$, and $\nu\in [-1/2,1/2]$. 
}
\end{lem}

\begin{lem}\label{lem-2}
Let $F$ be an even entire function of exponential type $\tau>0$ bounded
on~$\mathbb{R}$. Let $\Omega$ be an even entire function of finite exponential
type, all the zeroes of $\Omega$ be zeros of $F$, and, for some $m\in
\mathbb{Z}_+$,
\[
\liminf_{y\to +\infty}e^{-\tau y}y^{2m}|\Omega(iy)|>0.
\]
Then the function $F(z)/\Omega(z)$ is an even polynomial of degree at most $2m$.
\end{lem}

Lemma \ref{lem-2} is an easy consequence of Akhiezer's result \cite[Appendix
VII.10]{Le80}.

\section{
Logan problem for the Hankel transform}\label{sec-LH}

Let $\alpha\ge-1/2$ and $m\in\mathbb{Z}_{+}$. In this section we solve the
generalized Logan problem (with parameter $m$) for {the Hankel transform
$\mathcal{H}_{\alpha}$ in} $(\mathbb{R}_{+},d\nu_{\alpha}(\lambda))$. This is
the crucial step to prove Theorems \ref{thm-F} and \ref{thm-FU}.

{Consider the class $\mathcal{E}_{\alpha}(\mathbb{R}_{+})$ of
real-valued even entire functions $f$ of finite exponential type such that}
\begin{equation}\label{sigma}
f(\lambda)=\int_{0}^{\tau(f)}j_{\alpha}(\lambda t)\,d\sigma(t),
\end{equation}
where $\sigma$ is a function of bounded variation.

{Let} $\lambda(f)=\sup\{\lambda>0\colon f(\lambda)>0\}$. For $m\in
\mathbb{Z}_{+}$, denote by $\mathcal{E}_{\alpha,m}(\mathbb{R}_{+})$ the subclass of functions
$f\in \mathcal{E}_{\alpha}(\mathbb{R}_{+})$ such that
$\lambda((-1)^{m}f)<\infty$ and, if $m\ge 1$, 
$f\in L^1(\mathbb{R}_+,\lambda^{2m-2}\,d\nu_{\alpha})$ 
{and for} $k=0,\dots,m-1$
\begin{equation}\label{orth}
B_{\alpha}^{k}\mathcal{H}_{\alpha}(f)(0)=(-1)^{k}
\int_{0}^{\infty}\lambda^{2k}f(\lambda)\,d\nu_{\alpha}(\lambda)=0.
\end{equation}
We will see that this class is not empty, in particular,
 $f_{\alpha,m}(\lambda)=j_{\alpha}(\lambda)g_{\alpha,m}(\lambda)\in \mathcal{E}_{\alpha,m}(\mathbb{R}_{+})$,
  see \eqref{def-fam} and \eqref{def-gam}.
Due to \eqref{eq11}, for the Hankel transforms of functions
$f_{\alpha,m}$ and $g_{\alpha,m}$ one has
\[
\mathcal{H}_{\alpha}(f_{\alpha,m})=T_{\alpha}^{1}\mathcal{H}_{\alpha}(g_{\alpha,m}).
\]


\begin{thm}\label{thm-H}
\textup{(i)} {Let $f\in \mathcal{E}_{\alpha,m}(\mathbb{R}_{+})\setminus
\{0\}$ be given by \eqref{sigma}
such that  $\sigma$  is non-decreasing in some
neighborhood of the origin.} Then
\begin{equation}\label{H-1}
f\in L^1(\mathbb{R}_+,\lambda^{2m}\,d\nu_{\alpha}),\quad
(-1)^{m}\int_{0}^{\infty}\lambda^{2m}f(\lambda)\,d\nu_{\alpha}(\lambda)\ge
0,
\end{equation}
and
\begin{equation}\label{H-2}
2q_{\alpha,m+1}\le \lambda((-1)^{m}f)\tau(f).
\end{equation}
{Moreover,}
inequality \eqref{H-2} is sharp and the function $f_{\alpha,m}$ is the unique
extremizer up to a positive constant.

\smallbreak \textup{(ii)} {The functions $g_{\alpha,m}$ and
$f_{\alpha,m}$ are positive definitive} and
\begin{equation}\label{orth-fam}
\int_{0}^{\infty}\lambda^{2m}f_{\alpha,m}(\lambda)\,d\nu_{\alpha}(\lambda)=0.
\end{equation}

\textup{(iii)} {There holds
$g_{\alpha,m}=\mathcal{H}_{\alpha}(p_{\alpha,m}\chi_{[0,1]})$, where
$p_{\alpha,m}(t)$ is decreasing on $[0,1]$ and has a zero of multiplicity $2m+1$ at
$t=1$.}
\end{thm}

\begin{proof}
The proof is divided into several steps. Since the class $\mathcal{E}_{\alpha,m}(\mathbb{R}_{+})$ and
 the quantity $\lambda((-1)^{m}f)\tau(f)$ are invariant under dilations,
we let  for convenience  $\tau(f)=2$. {We also denote $q_{k}=q_{\alpha,k}$
for $k\ge 1$.}

\subsection*{Proof of \eqref{H-1}}
Let $m=0$. The embedding $\mathcal{E}_{\alpha,0}(\mathbb{R}_{+})\subset
L^1(\mathbb{R}_+,d\nu_{\alpha})$ can be shown using the method of Logan, see
\cite[Lemma]{Lo83b}.

{We consider the  positive definite kernel
$\varphi_{\varepsilon}(x)=j_{\alpha+1}^{2}(\varepsilon |x|/2)$,
$\varepsilon>0$. By \eqref{eq6}, \eqref{j-bound}, and
\eqref{j2-conv}, $\varphi_{\varepsilon}\in C(\mathbb{R}_{+})\cap
L^1(\mathbb{R}_+,d\nu_{\alpha})$,
$\|\varphi_{\varepsilon}\|_{C(\mathbb{R}_{+})}=\varphi_{\varepsilon}(0)=1$, and
$\mathrm{supp}\,\mathcal{H}_{\alpha}(\varphi_{\varepsilon})\subset
[0,\varepsilon]$.}
Since $d\sigma\ge 0$ in some neighborhood of the origin, then for sufficiently
{small $\varepsilon$} we have
\begin{align*}
0&\le \int_{0}^{\varepsilon}\mathcal{H}_{\alpha}(\varphi_{\varepsilon})(t)\,d\sigma(t)=
\int_{0}^{\infty}\mathcal{H}_{\alpha}(\varphi_{\varepsilon})(t)\,d\sigma(t)=
\int_{0}^{\infty}f(\lambda)\varphi_{\varepsilon}(\lambda)\,d\nu_{\alpha}(\lambda)\\
&=\int_{0}^{\lambda(f)}f(\lambda)\varphi_{\varepsilon}(\lambda)\,d\nu_{\alpha}(\lambda)-
\int_{\lambda(f)}^{\infty}|f(\lambda)|\varphi_{\varepsilon}(\lambda)\,d\nu_{\alpha}(\lambda),
\end{align*}
which implies
\[
\int_{\lambda(f)}^{\infty}|f(\lambda)|\varphi_{\varepsilon}(\lambda)\,d\nu_{\alpha}(\lambda)\le
\int_{0}^{\lambda(f)}f(\lambda)\varphi_{\varepsilon}(\lambda)\,d\nu_{\alpha}(\lambda)\le
\int_{0}^{\lambda(f)}|f(\lambda)|\,d\nu_{\alpha}.
\]
Letting $\varepsilon\to 0$, Fatou's lemma yields
\[
\int_{\lambda(f)}^{\infty}|f(\lambda)|\,d\nu_{\alpha}(\lambda)\le
\int_{0}^{\lambda(f)}|f(\lambda)|\,d\nu_{\alpha}(\lambda)<\infty,
\]
{which implies $f\in L^1(\mathbb{R}_+,d\nu_{\alpha})$.}

{Let $m\ge 1$. 
We have $f\in L^1(\mathbb{R}_+,d\nu_{\alpha})$ and 
$d\sigma(t)=\mathcal{H}_{\alpha}(f)(t)\,dt$, where
$\mathcal{H}_{\alpha}(f)(t)$ is continuous and nonnegative in some neighborhood of the origin.
Moreover, $(-1)^{m+1}f(\lambda)=|f(\lambda)|$ for $\lambda\ge \lambda((-1)^{m}f)$.}

Consider
\[
\rho_{\varepsilon}(\lambda)=\frac{(2m)!\,
\psi_{m}(\varepsilon\lambda)}{\varepsilon^{2m}\psi_{m}^{(2m)}(0)},
\]
where $\psi_m(\lambda)$ is given in \eqref{eq8}. We have 
\begin{equation}\label{eq19}
\psi_{m}^{(2m)}(0)>0,\quad\rho_{\varepsilon}(\lambda)\ge 0,\quad
\lim\limits_{\varepsilon\to
0}\rho_{\varepsilon}(\lambda)=\lambda^{2m},\quad \lambda\in\mathbb{R}_+.
\end{equation}
{In light of
\[
|\rho_{\varepsilon}(\lambda)-\lambda^{2m}|\le\frac{(2m)!}{\psi_{m}^{(2m)}(0)}\,\varepsilon^2 e^{\lambda^2/4}
\]
we derive that  $\rho_{\varepsilon}(\lambda)$ converges uniformly to  $\lambda^{2m}$  on any finite interval $[0,b]$
as $\varepsilon\to 0$.}

Taking into account \eqref{eq8}, \eqref{eq19}, the orthogonality condition \eqref{orth}, and
{nonnegativity of $\mathcal{H}_{\alpha}(f)$}
near the origin, we obtain
\begin{align}
(-1)^{m}\int_{0}^{\infty}\rho_{\varepsilon}(\lambda)f(\lambda)\,d\nu_{\alpha}(\lambda)&=
\frac{(2m)!}{\varepsilon^{2m}\psi_{m}^{(2m)}(0)}
\int_{0}^{\infty}f(\lambda)j_{\alpha}(\varepsilon\lambda)\,d\nu_{\alpha}(\lambda)
\notag\\
&=\frac{(2m)!}{\varepsilon^{2m}\psi_{m}^{(2m)}(0)}\,\mathcal{H}_{\alpha}(f)(\varepsilon)\ge
0. \label{eq21}
\end{align}
Thus,
\begin{equation}\label{eq19--}
(-1)^{m+1}\int_{\lambda((-1)^{m}f)}^{\infty}\rho_{\varepsilon}(\lambda)f(\lambda)\,d\nu_{\alpha}(\lambda)\le
(-1)^{m}\int_{0}^{\lambda((-1)^{m}f)}\rho_{\varepsilon}(\lambda)f(\lambda)\,d\nu_{\alpha}(\lambda).
\end{equation}
Using \eqref{eq21}, \eqref{eq19}, and Fatou's lemma we arrive at
\begin{align*}
&(-1)^{m+1}\int_{\lambda((-1)^{m}f)}^{\infty}\lambda^{2m}f(\lambda)\,d\nu_{\alpha}(\lambda)
=(-1)^{m+1}\int_{\lambda((-1)^{m}f)}^{\infty}\lim_{\varepsilon\to
0}\rho_{\varepsilon}(\lambda)f(\lambda)\,d\nu_{\alpha}(\lambda)\\
&\quad \le\liminf_{\varepsilon\to
0}{}(-1)^{m+1}\int_{\lambda((-1)^{m}f)}^{\infty}\rho_{\varepsilon}(\lambda)f(\lambda)\,d\nu_{\alpha}(\lambda).
\intertext{In light of \eqref{eq19--}, we continue as follows}
&\quad \le\liminf_{\varepsilon\to
0}{}(-1)^{m}\int_0^{\lambda((-1)^{m}f)}\rho_{\varepsilon}(\lambda)f(\lambda)\,d\nu_{\alpha}(\lambda)\\
&\quad =(-1)^{m}\int_0^{\lambda((-1)^{m}f)}\lim_{\varepsilon\to
0}\rho_{\varepsilon}(\lambda)f(\lambda)\,d\nu_{\alpha}(\lambda)
=(-1)^{m}\int_0^{\lambda((-1)^{m}f)}\lambda^{2m}f(\lambda)\,d\nu_{\alpha}(\lambda)<\infty,
\end{align*}
which gives \eqref{H-1}.

\subsection*{Proof of \eqref{H-2}}
Let $f\in \mathcal{E}_{\alpha,m}(\mathbb{R}_{+})$. We will prove that $q_{m+1}\le \lambda((-1)^{m}f)$.
Assume the converse, i.e.,
$\lambda((-1)^{m}f)<q_{m+1}$. {We have $(-1)^{m}f(\lambda)\le 0$ for $\lambda\ge
q_{m+1}$.}
By \eqref{H-1} we have $\lambda^{2m}f\in
\mathcal{B}_{\alpha}^2(\mathbb{R}_{+})$.
Then using Gauss' quadrature formula
\eqref{q-G} and \eqref{orth}, we get
\begin{align}
0&\le(-1)^{m}\int_{0}^{\infty}\lambda^{2m}f(\lambda)\,d\nu_{\alpha}(\lambda)=
(-1)^{m}\int_{0}^{\infty}\prod_{k=1}^{m}(\lambda^2-q_{k}^2)
f(\lambda)\,d\nu_{\alpha}(\lambda) \notag\\
&=
(-1)^{m}\sum_{s=m+1}^{\infty}
\gamma_{s}f(q_{s})\prod_{k=1}^{m}(q_{s}^2-q_{k}^2)\le 0.
\label{eq22}
\end{align}
Therefore, 
$q_{s}$,
 $s\ge m+1$, are zeros of multiplicity $2$ for $f$.
Similarly, applying
Gauss' quadrature formula for $f$, we obtain
\begin{equation}\label{eq23}
0=\int_{0}^{\infty}
\prod_{\substack{k=1\\ k\ne s}}^{m}(\lambda^2-q_{k}^2)f(\lambda)\,d\nu_{\alpha}(\lambda)=
\gamma_{s}\prod_{\substack{k=1\\ k\ne s}}^{m}(q_{s}^2-q_{k}^2)f(q_{s}),\quad
s=1,\dots,m.
\end{equation}
Therefore, $q_{s}$, $s=1,\dots,m$, are zeros of $f$.

{Take the function $\omega_{\alpha}(\lambda)$ from Lemma \ref{lem-1} and consider 
the following even functions of exponential type $4$:}
\[
F(\lambda)=\omega_{\alpha}(\lambda)f(\lambda),\quad
\Omega(\lambda)=\frac{\omega_{\alpha}(\lambda)j_{\alpha}^2(\lambda)}
{\prod_{k=1}^{m}(1-\lambda^2/q_{k}^2)}.
\]
{Note that $F\in L^{1}(\mathbb{R})$ since $f\in
L^1(\mathbb{R}_+,\lambda^{2\alpha+1}\,d\lambda)$ and
$\omega_{\alpha}(\lambda)\asymp \lambda^{2\alpha+1}$, $\lambda\to +\infty$.
Then $F$ is bounded on $\mathbb{R}$.}

From \eqref{eq6} and Lemma \ref{lem-1} we have $|\Omega(iy)|\asymp
y^{-2m}e^{4y}$ as $y\to +\infty$. Since all zeros of $\Omega(\lambda)$ are
also zeros of $F(\lambda)$, taking into account Lemma \ref{lem-2}, we obtain
\[
f(\lambda)=\frac{j_{\alpha}^2(\lambda)\sum_{k=0}^{m}c_k\lambda^{2k}}
{\prod_{k=1}^{m}(1-\lambda^2/q_{k}^2)},
\]
where $c_{k}\neq 0$ for some $k$. Note that $j_{\alpha}(\lambda)\notin
L^{2}(\mathbb{R}_+,d\nu_{\alpha})$, see \eqref{eq6}. This contradicts $f\in
L^1(\mathbb{R}_+,\lambda^{2m}\,d\nu_{\alpha})$. Hence, $\lambda((-1)^{m}f)\ge q_{m+1}$
and $\lambda((-1)^{m}f)\tau(f)\ge 2q_{m+1}$.

\smallbreak Now we consider the function $f_{\alpha,m}$ given by
\eqref{def-fam}. Note that in virtue of the estimate
$f_{\alpha,m}(\lambda)=O(\lambda^{-2\alpha-2m-3})$
as $\lambda\to \infty$ we
have $f_{\alpha,m}\in L^1(\mathbb{R}_+,\lambda^{2m}\,d\nu_{\alpha})$. Moreover,
$\tau(f_{\alpha,m})=2$
 and $\lambda((-1)^{m}f_{\alpha,m})=q_{\alpha,m+1}$.
Part~(i) is proved.

\smallbreak
To verify part~(ii), we first note that
Gauss' quadrature formula implies \eqref{orth-fam}.
To show the positive definiteness of $f_{\alpha,m}$, it is enough to prove that $g_{\alpha,m}$ is positive definite.

\subsection*{Positive definiteness of the function $g_{\alpha,m}$}
This result has been recently obtained by Cohn and de~Courcy-Ireland
\cite{CC18} for $\alpha=d/2-1$, $d\in \mathbb{N}$. We prove the same statement for any $\alpha$. For this, we
calculate the Hankel transform of $g_{\alpha,m}$ and show that it is nonnegative.

For fixed $\lambda_1,\dots,\lambda_k\in \mathbb{R}$, consider the polynomial 
\[
\omega_k(\lambda)=\omega(\lambda,\lambda_1,\dots,\lambda_k)=\prod_{i=1}^k(\lambda_i-\lambda), \quad \lambda\in \mathbb{R}.
\]
Then
\[
\frac{1}{\omega_k(\lambda)}=\sum_{i=1}^k\frac{1}{\omega_k'(\lambda_i)(\lambda_i-\lambda)}.
\]
Setting $\lambda_i=q_i^2$, we have
\begin{equation}\label{eq24}
\frac{1}{\prod_{i=1}^{k}(1-\lambda^2/q_{i}^2)}
=
\prod_{i=1}^{k}q_{i}^2 \frac{1}{\omega_{k}(\lambda^2)}
=
\prod_{i=1}^{k}q_{i}^2\sum_{i=1}^{k}\frac{1}{\omega_{k}'(q_{i}^2)(q_{i}^2-\lambda^2)}
=
\sum_{i=1}^{k}\frac{A_i}{q_{i}^2-\lambda^2},
\end{equation}
where
\begin{equation}\label{eq25}
\omega_{k}'(q_i^2)=\prod_{\substack{j=1\\ j\ne i}}^{k}(q_{j}^2-q_{i}^2),\quad A_i=
\frac{\prod_{{j=1}}^{k}q_{j}^2}
{\omega_k'(q_{i}^2)}.
\end{equation}
Note that
\begin{equation}\label{eq26}
\mathrm{sign}\,A_i=(-1)^{i-1}.
\end{equation}

Setting
\[
\varphi_{i}(t)=j_{\alpha}(q_{i}t),\quad i=1,\dots,m+1,
\]
we remark that $\varphi_{i}(t)$ are eigenfunctions
and $q_{i}^2$ are eigenvalues of the following
 Sturm--Liouville problem on $[0,1]$:
\begin{equation}\label{eq27}
(t^{2\alpha+1}u')'+\lambda^2t^{2\alpha+1}u=0,\quad u'(0)=0,\quad u(1)=0.
\end{equation}

It follows from \eqref{j-int}, \eqref{j-diff-rec}, and $j_{\alpha}(q_i)=0$ that
\[
\int_{0}^{\infty}\varphi_i(t)\chi_{[0,1]}(t)j_{\alpha}(\lambda t)t^{2\alpha+1}\,dt=
\int_0^1j_{\alpha}(q_it)j_{\alpha}(\lambda t)t^{2\alpha+1}\,dt=
-\frac{\varphi_i'(1)j_{\alpha}(\lambda)}{q_i^2-\lambda^2},
\]
or, equivalently,
\begin{equation}\label{eq28}
\mathcal{H}_{\alpha}\Bigl(-b_{\alpha}^{-1}\,\frac{\varphi_i\chi_{[0,1]}}{\varphi_i'(1)}\Bigr)(\lambda)=
\frac{j_{\alpha}(\lambda)}{q_{i}^2-\lambda^2}.
\end{equation}
Note that
\begin{equation}\label{eq29}
\mathrm{sign}\,\varphi_i'(1)=(-1)^i.
\end{equation}

Consider the following polynomial in eigenfunctions $\varphi_i(t)$:
\begin{equation}\label{eq30}
p_{\alpha,m}(t)=-b_{\alpha}^{-1}\sum_{i=1}^{m+1}\frac{A_i}{\varphi_i'(1)}\,\varphi_i(t)=
\sum_{i=1}^{m+1}B_i\varphi_i(t).
\end{equation}
Due to \eqref{eq26}, \eqref{eq27}, and \eqref{eq29}, we have that $B_i>0$,
$p_{\alpha,m}(0)>0$, and $p_{\alpha,m}(1)=0$. Moreover, in virtue of \eqref{eq24} and \eqref{eq28},
\begin{equation}\label{p-g}
g_{\alpha,m}(\lambda)=
\frac{j_{\alpha}(\lambda)}{\prod_{i=1}^{m+1}(1-\lambda^2/q_{i}^2)}=
\mathcal{H}_{\alpha}(p_{\alpha,m}\chi_{[0,1]})(\lambda).
\end{equation}
From this, it is enough to show that $p_{\alpha,m}(t)\ge 0$ on $[0,1]$. 

Define the Vandermonde determinant
$
\Delta(\lambda_1,\dots,\lambda_k)=\prod_{1\le j<i\le k}^{k}(\lambda_i-\lambda_j),
$ 
  then
\[
\frac{\Delta(\lambda_1,\dots,\lambda_k)}{\omega_k'(\lambda_i)}=
(-1)^{i-1}\Delta(\lambda_1,\dots,\lambda_{i-1},\lambda_{i+1},\dots,\lambda_k).
\]
In virtue of  \eqref{eq24} and \eqref{eq25}, we have
\begin{align}
p_{\alpha,m}(t)&=-c\sum_{i=1}^{m+1}(-1)^{i-1}\Delta(q_1^2,\dots,q_{i-1}^2,q_{i+1}^2,\dots,q_{m+1}^2)\,
\frac{\varphi_i(t)}{\varphi_i'(1)} \notag\\
&=-c\begin{vmatrix}
\frac{\varphi_1(t)}{\varphi_1'(1)}&\dots&\frac{\varphi_{m+1}(t)}{\varphi_{m+1}'(1)}\\
1&\dots&1\\
q_1^2&\dots &q_{m+1}^2\\
\hdotsfor[2]{3}\\
q_1^{2m-2}&\dots &q_{m+1}^{2m-2}
\end{vmatrix},\quad c=\frac{b_{\alpha}^{-1}\prod_{j=1}^{m+1}q_{j}^{2}}{\Delta(q_{1}^{2},\dots,q_{m+1}^{2})}>0.
\label{eq31}
\end{align}
Here and in what follows if $m=0$ we deal with only
the $(1,1)$ entries of the matrices.


Let us show that
\begin{equation}\label{eq32}
\begin{vmatrix}
\frac{\varphi_1(t)}{\varphi_1'(1)}&\dots&\frac{\varphi_{m+1}(t)}{\varphi_{m+1}'(1)}\\
1&\dots&1\\
q_1^2&\dots &q_{m+1}^2\\
\hdotsfor[2]{3}\\
q_1^{2m-2}&\dots &q_{m+1}^{2m-2}
\end{vmatrix}=(-1)^{\frac{m(m-1)}{2}}
\begin{vmatrix}
\frac{\varphi_1(t)}{\varphi_1'(1)}&\dots&\frac{\varphi_{m+1}(t)}{\varphi_{m+1}'(1)}\\
1&\dots&1\\
\frac{\varphi_1^{(3)}(1)}{\varphi_1'(1)}&\dots&\frac{\varphi_{m+1}^{(3)}(1)}{\varphi_{m+1}'(1)}\\
\hdotsfor[2]{3}\\
\frac{\varphi_1^{(2m-1)}(1)}{\varphi_1'(1)}&\dots&\frac{\varphi_{m+1}^{(2m-1)}(1)}{\varphi_{m+1}'(1)}
\end{vmatrix}.
\end{equation}
By \eqref{eq27}, we derive
\[
t\varphi_i''(t)+(2\alpha+1)\varphi_i'(t)+q_i^2t\varphi_i(t)=0.
\]
Therefore,
\[
t\varphi_i^{(s+2)}(t)+s\varphi_i^{(s+1)}(t)+(2\alpha+1)\varphi_i^{(s+1)}(t)+
q_i^2t\varphi_i^{(s)}(t)+sq_i^2\varphi_i^{(s-1)}(t)=0,
\]
which implies for $t=1$
\[
\varphi_i^{(s+2)}(1)=-(s+2\alpha+1)\varphi_i^{(s+1)}(1)-
q_i^2\varphi_i^{(s)}(1)-sq_i^2\varphi_i^{(s-1)}(1),\quad
\varphi_i^{(0)}(1)=0.
\]
By induction we then obtain
  for $k=0,1,\dots$
\[
\varphi_i^{(2k+1)}(1)=\varphi_i'(1)\sum_{j=0}^{k}a_{kj}(\alpha)q_{i}^{2j},\quad
\varphi_i^{(2k+2)}(1)=\varphi_i'(1)\sum_{j=0}^{k}b_{kj}(\alpha)q_{i}^{2j},
\]
where $a_{kj}(\alpha)$, $b_{kj}(\alpha)$ are polynomials in $\alpha$ with coefficients
not depending on $q_i$ and moreover $a_{kk}(\alpha)=(-1)^k$.
This implies for $k=1,2,\dots$
\begin{align}\label{eq34}
\frac{\varphi_i^{(2k)}(1)}{\varphi_i'(1)}&=\sum_{s=1}^{k}c_{0s}(\alpha)\,
\frac{\varphi_i^{(2s-1)}(1)}{\varphi_i'(1)},
\\
\label{eq35}
\frac{\varphi_i^{(2k+1)}(1)}{\varphi_i'(1)}&=\sum_{s=1}^{k}c_{1s}(\alpha)\,
\frac{\varphi_i^{(2s-1)}(1)}{\varphi_i'(1)}+(-1)^kq_i^{2k},
\end{align}
{where $c_{0s}(\alpha)$ and $c_{1s}(\alpha)$ do not depend on $q_i$.}
Then \eqref{eq35} implies \eqref{eq32} since 
\[
\begin{vmatrix}
\frac{\varphi_1(t)}{\varphi_1'(1)}&\dots&\frac{\varphi_{m+1}(t)}{\varphi_{m+1}'(1)}\\
1&\dots&1\\
\frac{\varphi_1^{(3)}(1)}{\varphi_1'(1)}&\dots&\frac{\varphi_{m+1}^{(3)}(1)}{\varphi_{m+1}'(1)}\\
\hdotsfor[2]{3}\\
\frac{\varphi_1^{(2m-1)}(1)}{\varphi_1'(1)}&\dots&\frac{\varphi_{m+1}^{(2m-1)}(1)}{\varphi_{m+1}'(1)}
\end{vmatrix}=
\begin{vmatrix}
\frac{\varphi_1(t)}{\varphi_1'(1)}&\dots&\frac{\varphi_{m+1}(t)}{\varphi_{m+1}'(1)}\\
1&\dots&1\\
-q_1^2&\dots &-q_{m+1}^2\\
\hdotsfor[2]{3}\\
(-1)^{m-1}q_1^{2m-2}&\dots &(-)^{m-1}q_{m+1}^{2m-2}
\end{vmatrix}.
\]

Further, taking into account \eqref{eq31} and \eqref{eq32}, we derive 
\begin{equation}\label{zv}
p_{\alpha,m}^{(1)}(1)=p_{\alpha,m}^{(3)}(1)=\cdots=p_{\alpha,m}^{(2m-1)}(1)=0.
\end{equation}
Therefore, by \eqref{eq30} and \eqref{eq34}, we obtain for $k=0,\dots,m$ that
\begin{align*}
p_{\alpha,m}^{(2k)}(1)&=-b_{\alpha}^{-1}\sum_{i=1}^{m+1}A_i\,\frac{\varphi_i^{(2k)}(1)}
{\varphi_i'(1)}=-b_{\alpha}^{-1}\sum_{i=1}^{m+1}A_i
\sum_{s=1}^{k}c_{0s}(\alpha)\,\frac{\varphi_i^{(2s-1)}(1)}
{\varphi_i'(1)}\\
&=-b_{\alpha}^{-1}\sum_{s=1}^{k}c_{0s}(\alpha)
\sum_{i=1}^{m+1}A_i\,\frac{\varphi_i^{(2s-1)}(1)}{\varphi_i'(1)}=
\sum_{s=1}^{k}c_{0s}(\alpha)p_{\alpha,m}^{(2s-1)}(1)=0.
\end{align*}
Together with (\ref{zv}), this implies that the zero $t=1$ of the polynomial
$p_{\alpha,m}(t)$ has multiplicity $2m+1$. Then taking into account \eqref{p-g}, the same also holds for $\mathcal{H}_{\alpha}(g_{\alpha,m})(t)$.

Let us show that $p_{\alpha,m}(t)$ does not have zeros on $[0,1)$ and hence
$p_{\alpha,m}(t)>0$ on $[0,1)$. This yields
that $g_{\alpha,m}$ is the positive definite function.

We use the facts that $\{\varphi_i(t)\}_{i=1}^{m+1}$ {for any $m\in
\mathbb{Z}_{+}$} is the Chebyshev system on the interval $(0,1)$ (see
Theorem~\ref{prop-bess} below) and any polynomial
$\sum_{i=1}^{m+1}c_{i}\varphi_{i}(t)$ on $(0,1)$ has at most $m$ zeros,
counting multiplicity.

We now consider the polynomial
\begin{equation}\label{eq37}
p(t,\varepsilon)=
\begin{vmatrix}
\frac{\varphi_1(t)}{\varphi_1'(1)}&\dots&\frac{\varphi_{m+1}(t)}{\varphi_{m+1}'(1)}\\
\frac{\varphi_1(1-\varepsilon)}{(-\varepsilon)\varphi_1'(1)}&\dots&
\frac{\varphi_{m+1}(1-\varepsilon)}{(-\varepsilon)\varphi_{m+1}'(1)}\\
\frac{\varphi_1(1-2\varepsilon)}{(-2\varepsilon)^3\varphi_1'(1)}&\dots&
\frac{\varphi_{m+1}(1-2\varepsilon)}{(-2\varepsilon)^3\varphi_{m+1}'(1)}\\
\hdotsfor[2]{3}\\
\frac{\varphi_1(1-m\varepsilon)}{(-m\varepsilon)^{2m-1}\varphi_1'(1)}&\dots&
\frac{\varphi_{m+1}(1-m\varepsilon)}{(-m\varepsilon)^{2m-1}\varphi_{m+1}'(1)}
\end{vmatrix}.
\end{equation}
{If $m=0$, it is positive on $(0,1)$ and, if $m\ge 1$,
for any $0<\varepsilon<1/m$, it has $m$ zeros at the points
$t_j=1-j\varepsilon$, $j=1,\dots,m$.}
Letting $\varepsilon\to
0+$, we note that
the polynomial $\lim\limits_{\varepsilon\to 0+}p(t,\varepsilon)$ does not have zeros on $(0,1)$.
Let us show that
\begin{equation}\label{zv1}
\lim\limits_{\varepsilon\to 0+}p(t,\varepsilon)=cp_{\alpha,m}(t)
\end{equation}
with some $c\ne 0.$ This implies that
 the polynomial $p_{\alpha,m}(t)$ is positive on $[0,1)$.

To show \eqref{zv1}, by
Taylor's theorem, we have
\[
\frac{\varphi_i(1-j\varepsilon)}{(-j\varepsilon)^{2j-1}\varphi_i'(1)}=
\sum_{s=1}^{2j-2}\frac{\varphi_i^{(s)}(1)}{s!\,(-j\varepsilon)^{2j-1-s}\varphi_i'(1)}+
\frac{\varphi_i^{(2j-1)}(1)+o(1)}{(2j-1)!\,\varphi_i'(1)},\quad \varepsilon\to
0,
\]
for $j=1,\dots,m-1$. Using formulas \eqref{eq34} and \eqref{eq35} and progressively subtracting the row $j$ from the row $j-1$
in the determinant \eqref{eq37}, we arrive at
\[
p(t,\varepsilon)=\frac{1}{\prod_{j=1}^{m-1}(2j-1)!}
\begin{vmatrix}
\frac{\varphi_1(t)}{\varphi_1'(1)}&\dots&\frac{\varphi_{m+1}(t)}{\varphi_{m+1}'(1)}\\
1+o(1)&\dots&1+o(1)\\
\frac{\varphi_1^{(3)}(1)+o(1)}{\varphi_1'(1)}&\dots&
\frac{\varphi_{m+1}^{(3)}(1)+o(1)}{\varphi_{m+1}'(1)}\\
\hdotsfor[2]{3}\\
\frac{\varphi_1^{(2m-1)}(1)+o(1)}{\varphi_1'(1)}&\dots&
\frac{\varphi_{m+1}^{(2m-1)}(1)+o(1)}{\varphi_{m+1}'(1)}
\end{vmatrix}.
\]
Then, taking into account \eqref{eq31} and \eqref{eq32}, we have \eqref{zv1}.

\subsection*{Monotonicity of 
   $p_{\alpha,m}$}
The polynomial $p(t,\varepsilon)$ vanishes at $m+1$ points:
$t_{j}=1-j\varepsilon$, $j=1,\dots,m$, and $t_{m}=1$, thus its derivative
$p'(t,\varepsilon)$ has $m$ zeros on the interval $(1-\varepsilon, 1)$.

In virtue of \eqref{j-diff-rec},
\[
\varphi_{i}'(t)=-\frac{q_{i}^{2}t}{2(\alpha+1)}\,j_{\alpha+1}(q_{i}t),\quad
t\in [0,1].
\]
This and Theorem~\ref{prop-bess} imply that 
$\{\varphi_i'(t)\}_{i=1}^{m+1}$ is the Chebyshev system on $(0,1)$. Therefore,
$p'(t,\varepsilon)$ does not have zeros on $(0,1-\varepsilon]$. Then for
$\varepsilon\to 0+$ we derive that $p_{\alpha,m}'(t)$ does not have zeros on
$(0,1)$. Since $p_{\alpha,m}(0)>0$ and $p_{\alpha,m}(1)=0$, then
$p_{\alpha,m}'(t)<0$ on $(0,1)$. Thus, $p_{\alpha,m}(t)$ is decreasing on the
interval $[0,1]$. This completes the proof of part~(iii).

\subsection*{Uniqueness of the extremizer $f_{\alpha,m}$}
As above, we will use Lemmas \ref{lem-1} and \ref{lem-2}. Let $f(\lambda)$ be
an extremizer and $\lambda((-1)^{m}f)=q_{m+1}$. Consider the functions
\[
F(\lambda)=\omega_{\alpha}(\lambda)f(\lambda),\quad
\Omega(\lambda)=\omega_{\alpha}(\lambda)f_{\alpha,m}(\lambda),
\]
where $f_{\alpha,m}$ is defined in \eqref{def-fam} and $\omega_{\alpha}$ is from Lemma \ref{lem-1}.

Note that all zeros of $\Omega(\lambda)$ are also zeros of $F(\lambda)$.
Indeed, we have $(-1)^{m}f(\lambda)\le 0$ for $\lambda\ge q_{m+1}$ and
$f(q_{m+1})=0$ (otherwise $\lambda((-1)^{m}f)<q_{m+1}$, which is a
contradiction). This and \eqref{eq22} imply that the points $q_{s}$, $s\ge
m+2$,
are double zeros of $f$. By \eqref{eq23}, we also have that $f(q_{s})=0$ for
$s=1,\dots,m$ and therefore the function $f$ has zeros (at~least, of order
one) at the points $q_{s}$, $s=1,\dots,m+1$.

Using asymptotic relations given in Lemma \ref{lem-1}, we derive that $F(\lambda)$ is the entire function of exponential type, integrable on real line and therefore bounded.
Taking into account \eqref{eq6} and Lemma \ref{lem-1}, we get
\[
|\Omega(iy)|\asymp y^{-2m-2}e^{4y},\quad y\to +\infty.
\]
Now using
Lemma \ref{lem-2}, we arrive at
$f(\lambda)=\psi(\lambda)f_{\alpha,m}(\lambda)$, where $\psi(\lambda)$ is an even polynomial of degree at most $2m+2$.
Note that the degree cannot be $2s$, $s=1,\dots,m+1$, since in this case \eqref{eq6} implies that $f\notin
L^1(\mathbb{R}_+,\lambda^{2m}\,d\nu_{\alpha})$. Thus,
$f(\lambda)=cf_{\alpha,m}(\lambda)$, $c>0$.
\end{proof}


\section{Uncertainty principle for bandlimited functions on $\mathbb{R}_{+}$}\label{uncer}


Let as above $\lambda(f)=\sup\{\lambda>0\colon f(\lambda)>0\}$,
$\mathcal{E}_{\alpha}(\mathbb{R}_{+})$ be the class of real-valued even
bandlimited functions $f\in C(\mathbb{R}_+)$, $\tau(f)$ be the type of a bandlimited function $f$, and {$B_{\alpha}$ denote the  operator \eqref{B-def}.}

Following the proof of Theorem \ref{thm-H}, we obtain the following uncertainty principle
{for bandlimited functions on $\mathbb{R}_{+}$}.

\begin{thm}\label{thm-HUG}
Let $\alpha\ge -1/2$ and $m,s\in \mathbb{Z}_+$.

\smallbreak
\textup{(i)} One has
\[
\inf \lambda((-1)^{m}f)\tau(f)=2q_{\alpha+s+1,m+1},
\]
where the infimum is taken over {all nontrivial functions} $f\in
\mathcal{E}_{\alpha}(\mathbb{R}_{+})\cap
L^1(\mathbb{R}_+,\lambda^{2m}\,d\nu_{\alpha})$ such that
\begin{equation}\label{sum-g-}
\begin{cases}
B_{\alpha}^{k}\mathcal{H}_{\alpha}(f)(0)=0,&k=0,\dots,m-1,\\
B_{\alpha}^{l}f(0)=0,&l=0,\dots,s-1,
\end{cases}
\end{equation}
and
\begin{equation}\label{sum-g--}
B_{\alpha}^{m}\mathcal{H}_{\alpha}(f)(0)\ge 0,\quad B_{\alpha}^{s}f(0)\le 0.
\end{equation}
Moreover, the function $\lambda^{2s+2}f_{\alpha+s+1,m}(\lambda)$
is the unique extremizer up to a positive constant,  which additionally satisfies
$B_{\alpha}^{m}\mathcal{H}_{\alpha}(f)(0)=B_{\alpha}^{s}f(0)=0$.

\textup{(ii)} One has
\[
\inf \lambda((-1)^{m}f)\tau(f)=2q_{\alpha+s,m+1},
\]
where the infimum is taken {over all nontrivial functions} $f\in
\mathcal{E}_{\alpha}(\mathbb{R}_{+})\cap
L^1(\mathbb{R}_+,\lambda^{2m+2s}\,d\nu_{\alpha})$ such that
\begin{equation}\label{sum-g---}
B_{\alpha}^{k}\mathcal{H}_{\alpha}(f)(0)=0,\quad k=s,\dots,m+s-1,\quad
B_{\alpha}^{m+s}\mathcal{H}_{\alpha}(f)(0)\ge 0.
\end{equation}
Moreover, the function
$f_{\alpha+s,m}(\lambda)$ is
the unique extremizer up to a positive constant,  which additionally satisfies $B_{\alpha}^{m+s}\mathcal{H}_{\alpha}(f)(0)=0$.
\end{thm}




\begin{proof}
{Part~(i).} \
Let $f$ be an admissible function. Without loss of generality we can assume that $\tau(f)=2$.
Unlike  the proof of Theorem \ref{thm-H} we will use the
Radau
 quadrature formula
\eqref{q-M}  with $r=s+1$.

First, we show that  $f^{(2l)}(0)=0$ for $0\le l\le s-1$ and $f^{(2s)}(0)\le
0$. Indeed, we have $B_{\alpha}\lambda^{2j}=2j(2\alpha+2j)\lambda^{2j-2}$,
and therefore for $j,l\in \mathbb{Z}_{+}$, by induction, we obtain for the $l$-th power of $B_{\alpha}$ that
$B_{\alpha}^{l}\lambda^{2j}=c_{\alpha,j,l}\lambda^{2j-2l}$, where
$c_{\alpha,j,l}>0$ for $j\ge l$ and $c_{\alpha,j,l}=0$ otherwise. This and  Taylor's expansion of  $f$ imply
\[
B_{\alpha}^{l}f(0)=\frac{c_{\alpha,l,l}}{(2l)!}\,f^{(2l)}(0).
\]

Second, let $\lambda((-1)^{m}f)<q_{m+1}'$,   where for simplicity we put
$q_{k}'=q_{\alpha+s+1,k}$, $k\ge 1$. Recall that $q_{\alpha,k}$ are zeros of
the Bessel function $j_{\alpha}(\lambda)$. Applying 
\eqref{q-M} to
$g(\lambda)=(-1)^{m}\prod_{k=1}^{m}(\lambda^2-{q_{k}'}^2) f(\lambda)$ (note that
  $g\in\mathcal{B}_{\alpha}^{2}(\mathbb{R}_{+})$), we derive
\[
\int_{0}^{\infty}g(\lambda)\,d\nu_{\alpha}(\lambda)=
\sum_{l=0}^{s}\alpha_{l,s+1}g^{(2l)}(0)+
\sum_{k=1}^{\infty}\gamma_{k,s+1}g(q_{k}').
\]
On the other hand, we have 
\[
\int_{0}^{\infty}g(\lambda)\,d\nu_{\alpha}(\lambda)=
(-1)^{m}\int_{0}^{\infty}\lambda^{2m}f(\lambda)\,d\nu_{\alpha}(\lambda)=
B_{\alpha}^{m}\mathcal{H}_{\alpha}(f)(0)\ge 0
\]
and
\[
g^{(2j)}(0)=0,\quad j=0,\dots,s-1,\quad
g^{(2s)}(0)=f^{(2s)}(0)\prod_{k=1}^{m}q_{k}'^2\le
0.
\]
Therefore,
\[
0\le \alpha_{s,s+1}g^{(2s)}(0)+
\sum_{k=m+1}^{\infty}\gamma_{k,s+1}g(q_{k}')\le 0,
\]
where we have used that $\gamma_{k,s+1}>0$
 and the fact that $g(\lambda)\le 0$ for $\lambda\ge
\lambda((-1)^{m}f)$. Thus, $f$ has double zeros at the points $q_{k}'$, $k\ge m+1$,
and {the zero of order $2s+2$ at the origin.}

Further, applying formula \eqref{q-M} for $j=1,\dots,m$ to the functions
$\prod_{\substack{k=1\\ k\ne j}}^{m}({\lambda^2-q_{k}'^2})f(\lambda)$, we
conclude that the function $f$ has at least simple zeros at the points $q_{j}$,
$1\le j\le m$.
Then as in the proof of Theorem \ref{thm-H}, using Lemmas \ref{lem-1}, \ref{lem-2} and the fact that
$\lambda^{2s+2}j_{\alpha+s+1}^{2}(\lambda)\notin L^{1}(\mathbb{R}_+,d\nu_{\alpha})$, {we derive that}
\[
f(\lambda)=\frac{\lambda^{2s+2}j_{\alpha+s+1}^{2}(\lambda)\sum_{k=0}^{m}c_k\lambda^{2k}}
{{\prod_{k=1}^{m}}(1-\lambda^{2}/q_{k}'^2)}\notin
L^1(\mathbb{R}_+,\lambda^{2m}\,d\nu_{\alpha}).
\]
Hence,
following  arguments similar to those used  to show (\ref{H-2}),
  we obtain that
 $\lambda((-1)^{m}f)\ge q_{m+1}'$. In fact, we have that $\lambda((-1)^{m}f)=q_{m+1}'$ for
\begin{equation}\label{f-a-s-1}
f(\lambda)=\frac{\lambda^{2s+2}j_{\alpha+s+1}^{2}(\lambda)}
{ {\prod_{k=1}^{m+1}}(1-\lambda^{2}/q_{k}'^2)}\in
L^1(\mathbb{R}_+,\lambda^{2m}\,d\nu_{\alpha}).
\end{equation}
Moreover, $f$ is a unique extremizer up to a positive constant (similarly to  the proof of the uniqueness of $f_{\alpha,m}$
 in Theorem \ref{thm-H}).

 Using \eqref{q-M} and $f^{(2s)}(0)=0$, we also have
 $B_{\alpha}^{m}\mathcal{H}_{\alpha}(f)(0)=B_{\alpha}^{s}f(0)=0$.

\smallbreak  Part~(ii). \
{The case $s=0$ follows from 
 Theorem
\ref{thm-H} since 
 to prove estimate \eqref{H-2}, we only used  condition \eqref{H-1}.}

Let $s\ge 1$.
We observe
that for any admissible  function $f$, that is, satisfying
condition (\ref{sum-g---}), the function $g(\lambda)=\lambda^{2s}f(\lambda)$ satisfies  conditions 
 (\ref{sum-g-}) and (\ref{sum-g--}) with the
parameter $s-1$ in place of $s$. 
At the same time, we have
$\lambda((-1)^{m}f)\tau(f)=\lambda((-1)^{m}g)\tau(g)$.
Hence, using  the fact that $c \lambda^{2s}f_{\alpha+s,m}(\lambda)$ is the unique extremizer in part~(i), we conclude that
$cf_{\alpha+s,m}(\lambda)$ is the unique extremizer  in  problem (ii).

\end{proof}

 \section{Number of zeros of positive definite function on $\mathbb{R}_+$
 }
 \label{sec-HU3}
It was proved in \cite{Lo83c}  that
if a function  from the class \eqref{eq1} has
$n$ zeros
on the interval $[0,L]$, then $L\ge \frac{\pi}{2}\,n$.
Moreover,
\[
F_{n}
(x)=\Bigl(\cos \frac{x}{n}\Bigr)^n
\]
is the unique extremal function.
 Note that the functions $F_{n}(\pi n(x-\frac{1}{2}))$ for $n=1$ and $3$
 coincide, up to constants, with the cosine Fourier transform of $f_0$ and $f_1$ (see Introduction)
on $[0,1]$.

In this section we study a similar problem for the Hankel transform $\mathcal{H}_{\alpha}$
with $\alpha\ge -1/2$. We will use the approach which was developed in Section~\ref{sec-LH}.
The key argument in the proof is based on
 the properties
of the polynomial $p_{\alpha,m}(t)$ defined in
\eqref{eq30}.

 Let $N_{I}(f)$ be the number of zeros of $f$ on $I$, counting multiplicity.
We will say that
 $f\in\mathcal{E}^+_\alpha(\mathbb{R}_{+})$~if 
$f(\lambda)=\int_{0}^{1}j_{\alpha}(\lambda t)\,d\sigma(t)$
with a nonnegative bounded
Stieltjes measure
 $d\sigma\ne 0$.

\begin{thm}\label{thm-HN}
{Let $\alpha\ge-1/2$, $n\in \mathbb{N}$, and
}\[
L(f,n)=\inf{}\{L>0\colon N_{[0,L]}(f)\ge n\}.
\]
Then
\[
\inf_{f\in\mathcal{E}^+_\alpha(\mathbb{R}_{+})} L(f,n)\le \theta_{\alpha,n}=
\begin{cases}
q_{\alpha,m+1},& n=2m+1,\\ q_{\alpha+1,m+1},& n=2m+2.
\end{cases}
\]
{Moreover, there exists a} 
function
$F_{\alpha,n}\in\mathcal{E}^+_\alpha(\mathbb{R}_{+})$ such that
$L(F_{\alpha,n},n)= \theta_{\alpha,n}$.


\end{thm}

\begin{rem}\label{newrem}
(1) For $\alpha=-1/2$, we have $q_{-1/2,m+1}=\frac{\pi}{2}\,(2m+1)$, $q_{1/2,m+1}=\pi (m+1)$,
{and, therefore,} $\theta_{-1/2,n}=\frac{\pi}{2}\,n$. 
 Hence,
we arrive at the  mentioned above result  \cite{Lo83c} 
\[
\inf_{f\in\mathcal{E}^+_{-1/2}(\mathbb{R}_{+})} L(f,n)=\frac{\pi}{2}\,n,
\]
where the extremal function $F_{-1/2,n}(\lambda)=(\cos \frac{\lambda}{n})^n$
has on $[0,\frac{\pi}{2}\,n]$ the unique zero $\lambda=\frac{\pi}{2}\,n$ of
multiplicity $n$.

(2) {We will show that the function} $F_{\alpha,n}(\lambda)$ has
 on $[0,\theta_{\alpha,n}]$  the
unique zero  $\lambda=\theta_{\alpha,n}$ of
multiplicity~$n$. Moreover, one has
 for $\lambda\in [0,\theta_{\alpha,n}]$
\[
F_{\alpha,n}(\lambda)=
\begin{cases}
p_{\alpha,m}(\lambda/q_{\alpha,m+1}),&n=2m+1,\\
\int_{\lambda/q_{\alpha+1,m+1}}^{1}sp_{\alpha+1,m}(s)\,ds,&n=2m+2.
\end{cases}
\]
\end{rem}


\begin{proof}
Let $n=2m+1$. Consider the polynomial (see \eqref{eq30})
\[
p_{\alpha,m}(t)=\sum_{i=1}^{m+1}B_ij_{\alpha}(q_{i}t),\quad t\in \mathbb{R}_{+},
\]
where $q_{i}=q_{\alpha,i}$. It has positive coefficients $B_i$ and the unique
zero $t=1$ of multiplicity $2m+1$ on the interval $[0,1]$ {(see Theorem \ref{thm-H} (iii))}. This and \eqref{p-g}
imply that the function
\[
F_{\alpha,n}(\lambda)=
\sum_{i=1}^{m+1}B_ij_{\alpha}\Bigl(\frac{q_{i}}{q_{m+1}}\,\lambda\Bigr),\quad
\lambda\in \mathbb{R}_{+},
\]
is the positive definite entire function of exponential type $1$  such
that $\lambda=q_{m+1}$ is a unique zero of multiplicity $2m+1$ on the interval
$[0,q_{m+1}]$. Therefore,
$L(F_{\alpha,n},2m+1)\le q_{m+1}$.

\smallbreak
Assume that $n=2m+2$. Consider the polynomial of type \eqref{eq30},
with respect to the parameter $\alpha+1$:
\[
p_{\alpha+1,m}(t)=\sum_{i=1}^{m+1}B_i'j_{\alpha+1}(q_{i}'t),\quad t\in
\mathbb{R}_{+},
\]
where $q_{i}'=q_{\alpha+1,i}$. As above, $B_i'>0$ and
\begin{equation}\label{B'-A'}
B_{i}'=-b_{\alpha+1}^{-1}\,\frac{A_i'}{\frac{d}{dt}\,j_{\alpha+1}(q_{i}'t)|_{t=1}},\quad
\sum_{i=1}^{m+1}\frac{A_i'}{q_{i}'^2-\lambda^2}=
\frac{1}{\prod_{i=1}^{m+1}(1-\lambda^2/q_{i}'^2)}.
\end{equation}
Set
\[
P(t)=\int_{t}^{1}sp_{\alpha+1,m}(s)\,ds=2(\alpha+1)\sum_{i=1}^{m+1}
\frac{B_{i}'}{q_{i}'^{2}}\,(j_{\alpha}(q_{i}'t)-j_{\alpha}(q_{i}')),
\]
where we have used \eqref{j-diff-rec}.

In virtue of \eqref{j1-diff},
$\frac{d}{dt}\,j_{\alpha+1}(q_{i}'t)\Bigr|_{t=1}=2(\alpha+1)j_{\alpha}(q_{i}')$
and therefore the polynomial $p_{\alpha+1,m}$ is positive and decreasing on $[0,1)$ and it
 has zero of multiplicity $2m+1$ at $t=1$. Then it is clear that the
polynomial $P(t)$ is positive and decreasing on
 $[0,1)$ and
 it has zero of multiplicity $2m+2$ at $t=1$.

Moreover, $P(t)$ can be represented as follows
\[
P(t)=B_{0}''+\sum_{i=1}^{m+1}B_{i}''j_{\alpha}(q_{i}'t),
\]
where $B_{i}''>0$ for $i\ge 1$ and, by \eqref{B'-A'},
\[
B_{0}''=-2(\alpha+1)\sum_{i=1}^{m+1}\frac{B_{i}'}{q_{i}'^{2}}\,j_{\alpha}(q_{i}')=
b_{\alpha+1}^{-1}\sum_{i=1}^{m+1}\frac{A_i'}{q_{i}'^{2}}=b_{\alpha+1}^{-1}>0.
\]

We finish the proof defining
\[
F_{\alpha,n}(\lambda)=
B_{0}''+\sum_{i=1}^{m+1}B_{i}''j_{\alpha}\Bigl(\frac{q_{i}'}{q_{m+1}'}\,\lambda\Bigr),\quad
\lambda\in \mathbb{R}_{+},
\]
which is a positive definite entire function of exponential type $1$, having
the unique zero $\lambda=q_{m+1}'$ of multiplicity $2m+2$ on $[0,q_{m+1}']$.
Therefore, $L(F_{\alpha,n},2m+2)\le q_{m+1}'$.
\end{proof}


\section{Generalized Logan problem for Dunkl and Fourier transforms }\label{sec-D}


In this section we solve the Logan problem for the Dunkl transform. We remark that in this case we will use the function
 $f_{\alpha,m}$ defined by \eqref{def-fam} for any
$\alpha\ge -1/2$ unlike the case of Fourier transform where we deal with only $\alpha=d/2-1$.


Basic facts on Dunkl harmonic analysis can be found in, e.g.,
\cite{Ro03}.
 Let a finite subset $R\subset \mathbb{R}^{d}\setminus\{0\}$ be a root system,
$G(R)\subset O(d)$ be a finite reflection group,
generated by reflections $\{\sigma_{a}\colon a\in R\}$, where $\sigma_{a}$ is a
reflection with respect to hyperplane $\<a,x\>=0$, and $\kappa\colon R\to
\mathbb{R}_{+}$ be a $G$-invariant multiplicity function.
{The Dunkl weight} is given by
\[
v_{\kappa}(x)=\prod_{a\in R_{+}}|\<a,x\>|^{2\kappa(a)},
\]
where $R_{+}$ positive subsystem of $R$.

Let $E_{\kappa}(x,y)$ be the symmetric Dunkl kernel associated with $G$ and $\kappa$
and $e_{\kappa}(x,y)=E_{\kappa}(x,iy)$ be the generalized exponential function.
{It is known that}
\[
e_\kappa(x,y)=\int_{\mathbb{R}^d}e^{i\<\xi,y\>}\,d\mu_x^\kappa(\xi),
\]
where $\mu_x^\kappa$ is a probability Borel measure
{supported on}
the convex hull of the $G$-orbit of $x\in \mathbb{R}^d$. Moreover, {one has
$(-\Delta_\kappa)^re_\kappa({\,\cdot\,},y)=|y|^{2r}e_\kappa({\,\cdot\,},y)$,
$r\in \mathbb{Z}_{+}$, where $\Delta_{\kappa}$ is the Dunkl Laplacian}.

Denote
\[
\alpha_{\kappa}=\frac{d}{2}-1+\sum_{a\in R_{+}}\kappa(a).
\]

We will need the following Fischer-type decomposition for the
Dunkl Laplacian: any even polynomial  $P(x)$, $x\in \mathbb{R}^d$, of degree at most  $2r$ can be represented by
\[
P(x)=\sum_{m=0}^{r}\sum_{j=0}^{m}|x|^{2m-2j}H_{m,2j}(x),
\]
where  $H_{m,2j}$ are even $\kappa$-harmonic homogeneous polynomials of degree $2j$, i.e., $\Delta_{\kappa}H_{m,2j}=0$
(see \cite[Sec.~5.1]{DX01}). Such polynomials satisfy
$$\Delta_{\kappa}|x|^{2i}H_{m,2j}(x)=2i(2i+4j+2\alpha_{\kappa})|x|^{2i-2}H_{m,2j}(x)$$
(see \cite[Lemma~5.1.9]{DX01}), which implies
\begin{equation}\label{D-diff}
\Delta_{\kappa}^{l}|x|^{2i}H_{m,2j}(x)=c_{ijl}|x|^{2i-2l}H_{m,2j}(x),\quad
c_{ijl}=0\quad \text{for}\quad i<l.
\end{equation}

The Dunkl transform is defined as follows
\[
\mathcal{F}_{\kappa}(f)(y)=c_{\kappa}\int_{\mathbb{R}^{d}}f(x)\overline{e_{\kappa}(x,y)}
v_{\kappa}(x)\,dx,\quad y\in \mathbb{R}^{d},
\]
where $c_{\kappa}^{-1}=\int_{\mathbb{R}^{d}}e^{-|x|^{2}/2}v_{\kappa}(x)\,dx $
is the Macdonald--Mehta--Selberg integral. It is a unitary operator in
$L^2(\mathbb{R}^d,d\mu_\kappa)$ such that
$\mathcal{F}_{\kappa}^{-1}(f)(x)=\mathcal{F}_{\kappa}(f)(-x)$.

In the non-weighted case ($\kappa=0$) we have $d\mu_{0}(x)=(2\pi)^{-d/2}\,dx$,
$e_{0}(x,y)=e^{i\<x,y\>}$, $\Delta_{0}=\Delta$, and $\mathcal{F}_{0}$ is the
Fourier transform.

Let $f\in C(\mathbb{R}^{d})$ be such that
\begin{equation}\label{repres}
f(x)=\int_{\mathbb{R}^{d}}e_{\kappa}(x,y)\,d\mu(y)
\end{equation}
with  a finite nonnegative Borel measure $\mu$.
We call such functions positive definite with respect to the Dunkl transform, if  $\mu$ is
nonnegative. For $\kappa=0$, by Bochner's theorem, we arrive at the usual concept of positive definiteness.




{
Denote  by  $\mathcal{E}_{\kappa}(\mathbb{R}^{d})$ the class of all even
{real-valued continuous bandlimited} functions $f$
of form
 \eqref{repres}
 {with the compactly supported measure $\mu$}. As usual,
 $\tau(f)$ is the  exponential
(spherical) type of $f$ if $\mathrm{supp}\,\mu\subset B_{\tau(f)}^{d}$ (cf.\ \cite{Je06}).
Recall that 
$\lambda(f)=\sup\{|x|>0\colon f(x)>0\}$.

We are now in a position to formulate the complete solution of the generalized Logan problem as well as
the uncertainty principle relations for {the Dunkl transform}.

\begin{thm}\label{thm-D}
Let $d\ge \mathbb{N}$ and $m,s\in \mathbb{Z}_+$.

\smallbreak \textup{(i)} One has
\[
\inf \lambda((-1)^{m}f)\tau(f)=2q_{\alpha_{\kappa},m+1},
\]
where the infimum is taken over all nontrivial functions $f\in
\mathcal{E}_{\kappa}(\mathbb{R}^{d})$ such that the measure $\mu$ in
\eqref{repres} is nonnegative in some {neighborhood of the origin}
  and,  if $m\ge 1$,
  $f\in L^1(\mathbb{R}^d,|x|^{2m-2}v_{\kappa}(x)\,dx)$
  {and the condition}
\[
\Delta_{\kappa}^{j}\mathcal{F}_{\kappa}(f)(0)=0,\quad
j=0,\dots,m-1,
\]
is fulfilled.
Moreover, the
positive definite  radial function $f_{\alpha_{\kappa},m}(|x|)$
is the unique extremizer
 up to a positive constant. 
{This function satisfies
$f\in L^1(\mathbb{R}^d,|x|^{2m}v_{\kappa}(x)\,dx)$ and
$\Delta_{\kappa}^{m}\mathcal{F}_{\kappa}(f)(0)=0$.
}

\smallbreak
\textup{(ii)} One has
\[
\inf \lambda((-1)^{m}f)\tau(f)=2q_{\alpha_{\kappa}+s+1,m+1},
\]
where the infimum is taken over all nontrivial functions $f\in
\mathcal{E}_{\kappa}(\mathbb{R}^{d})\cap
L^1(\mathbb{R}^d,|x|^{2m}v_{\kappa}(x)\,dx)$ such that
\begin{equation}\label{vsp-2}
\begin{cases}
\Delta_{\kappa}^{j}\mathcal{F}_{\kappa}(f)(0)=0,&j=0,\dots,m-1,\\
\Delta_{\kappa}^{l}f(0)=0,&l=0,\dots,s-1,
\end{cases}
\end{equation}
and
\[
\Delta_{\kappa}^{m}\mathcal{F}_{\kappa}(f)(0)\ge 0,\quad
\Delta_{\kappa}^{s}f(0)\le 0.
\]
{Moreover, each extremizer has the form $r(x)f_{\alpha_{\kappa}+s+1,m}(|x|)$
and satisfies the condition
$\Delta_{\kappa}^{m}\mathcal{F}_{\kappa}(f)=\Delta_{\kappa}^{s}f(0)=0$. Here
\[
r(x)=\sum_{j=0}^{s+1}|x|^{2s+2-2j}h_{2j}(x)\ge 0,\quad |x|\ge q_{\alpha_{\kappa}+s,m+1},
\]
where $h_{2j}(x)$ are even $\kappa$-harmonic polynomials of order at most $2j$
{ such that $h_{0}>0$}, $h_{2j}(0)=0$, $j=1,\dots,s+1$.}

\smallbreak
\textup{(iii)} One has
\[
\inf \lambda((-1)^{m}f)\tau(f)=2q_{\alpha_{\kappa}+s,m+1},
\]
where the infimum is taken over all nontrivial functions $f\in
\mathcal{E}_{\kappa}(\mathbb{R}^{d})\cap L^1(\mathbb{R}^d,|x|^{2m+2s}v_{\kappa}(x)\,dx)$ such that
\[
\Delta_{\kappa}^{j}\mathcal{F}_{\kappa}(f)(0)=0,\quad j=s,\dots,m+s-1,\quad
\Delta_{\kappa}^{m+s}\mathcal{F}_{\kappa}(f)(0)\ge 0.
\]
The function $f_{\alpha_{\kappa}+s,m}(|x|)$ is the unique extremizer up to a positive constant. 
{Moreover, this function satisfies} $\Delta_{\kappa}^{m+s}\mathcal{F}_{\kappa}(f)(0)=0$.
\end{thm}

\begin{rem}\label{remrem}
(1)
 For
$s=0$,  the class of admissible functions in part~(iii) of Theorem \ref{thm-D}
contains
admissible functions from part~(i).

(2)
For $\kappa=0$, part~(i) implies Theorem~\ref{thm-F}, part (ii) implies
Theorem~\ref{thm-FU}, and part~(iii) implies Remark \ref{rem-FU}.

(3) In part~(ii),  if a
polynomial $r(x)$ is nonnegative on $\mathbb{R}^d$, then it is an even
homogeneous polynomial of order $2s+2$.
\end{rem}

\begin{proof}
Our main idea is to reduce the proof of Theorem \ref{thm-D} to the case of
Hankel transform of radial functions.
 Using polar coordinates, we have
\begin{align*}
c_{\kappa}\int_{\mathbb{R}^{d}}f(x)v_{\kappa}(x)\,dx&=
\int_{0}^{\infty}\int_{\mathbb{S}^{d-1}}f(\lambda
x')\,c_{\kappa}v_{\kappa}(x')\,d\omega_\kappa(x')\,\lambda^{2\alpha_{\kappa}+1}\,d\lambda\\
&=\int_{0}^{\infty}\int_{\mathbb{S}^{d-1}}f(\lambda
x')\,d\omega_\kappa(x')\,d\nu_{\alpha_\kappa}(\lambda),
\end{align*}
where $d\nu_{\alpha_\kappa}$ is given by \eqref{def-b-nu},
$\mathbb{S}^{d-1}=\{x'\in\mathbb{R}^d\colon |x'|=1\}$ is the Euclidean sphere,
and $d\omega_\kappa(x')=b_{\alpha_{\kappa}}^{-1}c_{\kappa}v_\kappa(x')\,dx'$ is
{a probability measure on} $\mathbb{S}^{d-1}$ \cite[Sec.~2.2]{Ro03b}. In particular, for a radial
function $f(x)=f_{0}(|x|)$ one has
\begin{equation}\label{int-rad}
\mathcal{F}_{\kappa}(f)(0)=
c_{\kappa}\int_{\mathbb{R}^{d}}f(x)v_{\kappa}(x)\,dx=
\int_{0}^{\infty}f_{0}(\lambda)\,d\nu_{\alpha_\kappa}(\lambda).
\end{equation}

Let now $f\in \mathcal{E}_{\kappa}(\mathbb{R}^{d})$ be a function of type $\tau$, written
$f(x)=\int_{B_{\tau}^{d}}e_{\kappa}(x,y)\,d\mu(y)$.
We consider its radial part
$
f_{0}(\lambda)=\int_{\mathbb{S}^{d-1}}f(\lambda x')\,d\omega_{\kappa}(x').
$ 
 Due to the well-known formula \cite[Corollary~2.5]{Ro03b}
\[
\int_{\mathbb{S}^{d-1}}e_{\kappa}(\lambda
x',y)\,d\omega_{\kappa}(x')=j_{\alpha_{\kappa}}(\lambda |y|),\quad y\in \mathbb{R}^{d},
\]
we conclude that $f_{0}$ can be represented by
\begin{equation}\label{f0-repr}
f_{0}(\lambda)=\int_{B_{\tau}^{d}}j_{\alpha_{\kappa}}(\lambda |y|)\,d\mu(y)=
\int_{0}^{\tau}j_{\alpha_{\kappa}}(\lambda t)\,d\sigma(t),
\end{equation}
where $\sigma$ is a function of bounded variation. It is also clear that if
$d\mu$ in \eqref{f0-repr} is nonnegative in some neighborhood of the origin (or
everywhere), then $d\sigma$ satisfies the same property.

In light of \eqref{int-rad} and \eqref{f0-repr},
  we derive that
\begin{equation}\label{B-D}
\begin{gathered}
B_{\alpha_{\kappa}}^{r}\mathcal{H}_{\alpha_{\kappa}}(f_{0})(0)=
\Delta_{\kappa}^{r}\mathcal{F}_{\kappa}(f)(0)=(-1)^{r}
c_{\kappa}\int_{\mathbb{R}^{d}}|x|^{2r}f(x)v_{\kappa}(x)\,dx,\\
B_{\alpha_{\kappa}}^{r}f_{0}(0)=\Delta_{\kappa}^{r}f(0)=(-1)^{r}\int_{\mathbb{R}^{d}}|y|^{2r}\,d\mu(y).
\end{gathered}
\end{equation}

In virtue of these relationships we note { that} if a function $f$ is admissible in any of  problems  (i)--(iii) in Theorem \ref{thm-D}, then
its radial part
$f_{0}(|x|)$ is also admissible in the same problem and  $\lambda((-1)^{m}f_{0})\tau(f_{0})\le \lambda((-1)^{m}f)\tau(f)$.
 Hence, the corresponding infimums are attained on  radial functions.


{Formulas \eqref{f0-repr} and \eqref{B-D}} also imply that radial
extremizers in problems (i)--(iii) coincide with extremizers in Theorems
\ref{thm-H} and \ref{thm-HUG} for Hankel transforms. Thus, the functions
$f_{\alpha_{\kappa},m}(|x|)$, $|x|^{2s+2}f_{\alpha_{\kappa}+s+1,m}(|x|)$, and
$f_{\alpha_{\kappa}+s,m}(|x|)$ are extremizers for problems (i), (ii), and
(iii), respectively.

Note that for any admissible function $f$ from part~(i), taking into account Theorem \ref{thm-H}, we have that
$\Delta_{\kappa}^{m}\mathcal{F}_{\kappa}(f)(0)=B_{\alpha_{\kappa}}^{m}\mathcal{H}_{\alpha_{\kappa}}(f_{0})(0)\ge
0$. This implies part~(1) of Remark~\ref{remrem}.

It is left to prove the uniqueness of extremizers in problems (i)--(iii).

\smallbreak Part~(ii). { Let $\tau=2$,
$q_{j}'=q_{\alpha_{\kappa}+s+1,j}$, and $f$ be an extremizer. 
Then $(-1)^{m+1}f(x)\ge 0$ for $|x|\ge q_{m+1}'$ and its radial part is
\begin{equation}\label{eq-con1--}
f_{0}(\lambda)=c\lambda^{2s+2}f_{\alpha_{\kappa}+s+1,m}(\lambda), \quad c>0.
\end{equation}
Therefore,
$\int_{\mathbb{S}^{d-1}}f(q_{j}'x')\,d\omega_{\kappa}(x')=0$ for $j\ge m+1$,
which gives $f(x)=0$ if $|x|=q_{j}'$, 
$j\ge m+1$. Moreover, $f\in L^1(\mathbb{R}^d,|x|^{2m}v_{\kappa}(x)\,dx)$, since, in light of \eqref{f-a-s-1},
\[
c_{\kappa}\int_{|x|\ge
q_{m+1}'}|x|^{2m}|f(x)|v_{\kappa}(x)\,dx=(-1)^{m+1}c\int_{q_{m+1}'}^{\infty}
\lambda^{2m+2s+2}f_{\alpha_{\kappa}+s+1,m}(\lambda)\,d\nu_{\alpha_\kappa}(\lambda)<\infty.
\]

Denote $f(x)=f(\lambda x')=f_{x'}(\lambda)$, where $\lambda=|x|$, $x'=x/|x|$.
Since  $f$ is even and
$f_{x'}(\lambda)=\int_{B_{\tau}^{d}}e_{\kappa}(\lambda x',y)\,d\mu(y)$, then  $f_{x'}$ is the even entire function of exponential type  $\tau$ bounded
on $\mathbb{R}$. By Fubini's theorem,  $f_{x'}\in
L^{1}(\mathbb{R}_{+},\lambda^{2m}\,d\nu_{\alpha_{\kappa}})$.

The function  $f_{x'}(\lambda)$ keeps its sign for $\lambda\ge q_{m+1}'$ and
$f_{x'}(q_{j}')=0$ for $j\ge m+1$. Hence, $q_{j}'$ are double zeros for $j\ge m+2$.
Therefore, we have
\begin{equation}\label{f-r-f}
f_{x'}(\lambda)=r_{x'}(\lambda)f_{\alpha_{\kappa}+s+1,m}(\lambda)
\end{equation}
with some even entire function $r_{x'}(\lambda)$ of exponential type.
Similar to the proof of uniqueness  of extremizer  $f_{\alpha,m}$ in Theorem~\ref{thm-H}, using
Lemmas \ref{lem-1} and \ref{lem-2}, we obtain that
$r_{x'}(\lambda)$ is an even polynomial of degree at most $2s+2$ (otherwise  $f_{x'}\notin
L^{1}(\mathbb{R}_{+},\lambda^{2m}\,d\nu_{\alpha_{\kappa}})$).


Thus, by \eqref{f-r-f}, we have
\begin{equation}\label{frf}
f(x)=r(x)f_{\alpha_{\kappa}+s+1,m}(|x|),
\end{equation}
where $r(x)=\sum_{k=0}^{s+1}c_{k}(x')|x|^{2k}$.
Taylor's expansions are given by
\[
f(x)=\sum_{l=0}^{\infty}A_l(x')|x|^{2l},\quad f_{\alpha_{\kappa}+s+1,m}(|x|)=\sum_{j=0}^{\infty}a_j|x|^{2j},
\]
where $A_l(x')$ are homogeneous polynomials of degree  $2l$, $A_0(x')=A_0$, and 
$a_0=1$. Therefore, we arrive at the linear system
\[
\sum_{j=0}^lc_j(x')a_{l-j}=A_l(x'),\quad l=0,1,\dots,s+1,
\]
 in   variables $c_j(x')$.
We derive that 
\[
c_0=A_0,\quad c_j(x')=A_{j}(x')+\sum_{i=0}^{j-1}b_{ij}A_{i}(x'),\quad
j=1,\dots,l.
\]
{ Thus, $c_j(x')=A_{j}(x')+\sum_{i=0}^{j-1}b_{ij}A_{i}(x')|x|^{2j-2i}$} are
homogeneous polynomials of degree  $2j$, $j=1,\dots,l$, and then
$r(x)$ is an even polynomial of degree  $2s+2$.

\bigskip
Now we find
under which  conditions on $r$ the function $f$ is an extremizer.
 Since $\lambda((-1)^{m}f)=q_{m+1}'$, we necessarily have
\[
r(x)\ge 0,\quad |x|\ge q_{m+1}'.
\]
We write $r(x)=\sum_{k=0}^{s+1}r_{2k}(x)$, where $r_{2k}(x)$
are
homogeneous polynomials of degree  $2k$.
{By \eqref{frf} and \eqref{eq-con1--},}
\[
f_{0}(\lambda)=f_{\alpha_{\kappa}+s+1,m}(\lambda)
\sum_{k=0}^{s+1}\int_{\mathbb{S}^{d-1}}\lambda^{2k}r_{2k}(x')\,d\omega_{\kappa}(x')=
c\lambda^{2s+2}f_{\alpha_{\kappa}+s+1,m}(\lambda).
\]
{This implies}
\begin{equation}\label{eq-con2}
\int_{\mathbb{S}^{d-1}}r_{2k}(x')\,d\omega_{\kappa}(x')=0,
\quad k=0,1,\dots,s,\quad
\int_{\mathbb{S}^{d-1}}r_{2s+2}(x')\,d\omega_{\kappa}(x')>0.
%
\end{equation}
In particular, $r_0=0$. Furthermore,
 \eqref{eq-con2}, the
Fisher-type decomposition
$$r(x)=\sum_{j=0}^{s+1}|x|^{2s+2-2j}h_{2j}(x)$$ with
$h_{2j}(x)$ being even $\kappa$-harmonic polynomials of order at most $2j$, and the fact that
$h_{2j}(0)=\int_{\mathbb{S}^{d-1}}h_{2j}(x)\,d\omega_{\kappa}(x')$ imply that
\begin{equation}\label{eq-con3}
h_{0}>0,\quad h_{2j}(0)=0,\quad j=1,\dots,s+1.
\end{equation}
It is enough to verify that the function  $f(x)=r(x)f_{\alpha_{\kappa}+s+1,m}(|x|)$ is an extremizer.

Let us  show (\ref{vsp-2}). By Theorem \ref{thm-HUG},
for $k=0,1,\dots,m$ we have
\begin{align*}
&c_{\kappa}\int_{\mathbb{R}^{d}}|x|^{2k}f(x)v_{\kappa}(x)\,dx\\
&\quad =\sum_{j=0}^{s+1}\int_{0}^{\infty}\lambda^{2k+2s+2-2j}
f_{\alpha_{\kappa}+s+1,m}(\lambda)\,d\nu_{\alpha_\kappa}(\lambda)\int_{\mathbb{S}^{d-1}}h_{2j}(x)
v_{\kappa}(x')\,d\omega_\kappa(x')\\
&\quad =\sum_{j=0}^{s+1}h_{2j}(0)\int_{0}^{\infty}\lambda^{2k+2s+2-2j}
f_{\alpha_{\kappa}+s+1,m}(\lambda)\,d\nu_{\alpha_\kappa}(\lambda)\\
&\quad =h_0(0)\int_{0}^{\infty}\lambda^{2k+2s+2}
f_{\alpha_{\kappa}+s+1,m}(\lambda)\,d\nu_{\alpha_\kappa}(\lambda)=0.
\end{align*}
Since
\[
f(x)=\sum_{j=0}^{s+1}|x|^{2s+2-2j}h_{2j}(x)\sum_{k=0}^{\infty}c_k|x|^{2k}=\sum_{k=0}^{\infty}c_k\sum_{j=0}^{s+1}|x|^{2s+2-2j+2k}h_{2j}(x),
\]
both \eqref{D-diff} and \eqref{eq-con3} imply that
 $\Delta_{\kappa}^{l}f(0)=0$ for $l=0,1,\dots,s$. Thus,  condition  (\ref{vsp-2}) holds and moreover,
 $\Delta_{\kappa}^{m}\mathcal{F}_{\kappa}(f)=\Delta_{\kappa}^{s}f(0)=0$ is valid.

Finally, let us show that if $r(x)\ge 0$ on $\mathbb{R}^d$, then $r(x)$ is
homogeneous polynomials of degree   $2s+2$.
Assume that
 $r(x)=\lambda^{k_0}\sum_{k=k_{0}}^{s+1}\lambda^{k-k_{0}}r_{2k}(x')$,
 $x=\lambda x'$, where
 $1\le k_{0}\le s$
 and
$r_{2k_{0}}(x)\ne 0$ (recall that $r_0=0$). 
Using $\int_{\mathbb{S}^{d-1}}r_{2k_{0}}(x')\,d\omega_{\kappa}(x')=0$,
we derive  $r(\lambda x'_0)<0$ for some $x'_0\in \mathbb{S}^{d-1}$ and
sufficiently small  $\lambda>0$. This contradiction implies that  $r(x)=r_{2s+2}(x)$.


\smallbreak
\smallbreak Parts~(i) and~(iii) with  $s=0$.
Similar reasonings as above imply that any extremizer has the form
$cf_{\alpha_{\kappa},m}(|x|)$ with $c>0$.

\smallbreak Part~(iii) with $s\ge 1$. As in the proof of Theorem
\ref{thm-HUG}, we reduce the question about uniqueness of an extremizer $f$ in part~(iii) to similar problem in part~(ii) with $s-1$ in place of $s$.
Thus, we arrive at the function
$c
f_{\alpha_{\kappa}+s,m}(|x|)$, $c>0$.}

\end{proof}

\section{Chebyshev systems of normalized Bessel functions}\label{sec-cheb}

Recall that $N_{I}(f)$ stands for the number of zeros of $f$ on $I$, counting multiplicity.
A family of
 real-valued functions $\{\varphi_{k}(t)\}
 $
defined on an interval $I\subset \mathbb{R}$ is a Chebyshev system (T-system)
if for any $n\in \mathbb{N}$ and any nontrivial linear combination
$$P(t)=\sum_{k=1}^{n}A_{k}\varphi_{k}(t),$$ there holds $N_{I}(P)\le n-1$, see, e.g., \cite[Chap.~II]{Ac04}.

As above we assume that $\alpha\ge -1/2$, $q_{k}=q_{\alpha,k}$, and
$q_{k}'=q_{\alpha+1,k}$ for $k\in \mathbb{N}$. The main result of this section is the following theorem.

\begin{thm}\label{prop-bess}
\textup{(i)} The families of the Bessel functions
\begin{equation}\label{bess-sys}
\{j_{\alpha}(q_{k}t)\}_{k=1}^{\infty},\quad
\{1,j_{\alpha}(q_{k}'t)\}_{k=1}^{\infty}
\end{equation}
form Chebyshev systems on $[0,1)$ and $[0,1]$, respectively.

\textup{(ii)}~The families of the Bessel functions
\[
\{j_{\alpha+1}(q_{k}t)\}_{k=1}^{\infty},\quad
\{j_{\alpha}(q_{k}'t)-j_{\alpha}(q_{k}')\}_{k=1}^{\infty}
\]
form Chebyshev systems on $(0,1)$.
\end{thm}
For $\alpha=-1/2$ this theorem becomes the well-known result for trigonometric systems, which has many applications in approximation theory
(see \cite[Chap.~II]{Ac04}).
For $\alpha>-1/2$ this result seems to be new.

We will use the following Sturm's theorem on zeros of linear combinations of
eigenfunctions of Sturm--Liouville problem. This result is not widely known in
the literature, see the discussion in \cite{BH17}.


\begin{thm}[Sturm, 1836; Liouville, 1836]\label{thm-sturm}
Let $\{V_{k}\}_{k=1}^{\infty}$ be the system of eigenfunctions associated to
eigenvalues $\rho_{1}<\rho_{2}<\dots$ of the following Sturm--Liouville problem
on the interval $[a,b]$:
\begin{equation}\label{sturm-liuv-vsp}
(KV')'+(\rho G-L)V=0,\quad (KV'-hV)(a)=0,\quad (KV'+HV)(b)=0,
\end{equation}
where $G,K,L\in C[a,b]$, $K\in C^{1}(a,b)$, $K,G>0$ on $(a,b)$,
 $h,H\in [0,\infty]$ and $\rho$ denotes the spectral parameter.

Then for any nontrivial real polynomial of the form
\[
P=\sum_{k=m}^{n}A_{k}V_{k},\quad m,n\in \mathbb{N},\quad m\le n,
\]
we have
\[
m-1\le N_{(a,b)}(P)\le n-1.
\]
In particular, every $k$-th eigenfunction $V_{k}$ has exactly $k-1$ simple {zeros in $(a,b)$}.
\end{thm}

For trigonometric system this result is 
  known  as the Sturm--Hurwitz theorem (see, e.g., \cite{AHP05}).


Note that in the proof given by Liouville (see \cite{BH17}) it is enough to assume that
$K,G>0$ only on the interval $(a,b)$.
This allows us to include the singular case, that is, when
  $K$ and $G$ may have zeros at the endpoints of $[a,b]$. In particular, we may deal with the Sturm--Liouville problem for Bessel functions.

\begin{proof}[Proof of Theorem \ref{prop-bess}]
We will use the fact that, by Theorem \ref{thm-sturm}, the system of
eigenfunctions $\{V_{k}\}_{k=1}^{\infty}$ is the Chebyshev system. We note that
\eqref{bess-sys} are the families of eigenfunctions for the (singular for
$\alpha>-1/2$) Sturm--Liouville problem {(see \cite{Le51})}
\begin{equation}\label{sturm-liuv-bess}
\begin{gathered}
(t^{2\alpha+1}u'(t))'+\lambda^{2}t^{2\alpha+1}u(t)=0,\quad
t\in [0,1],\\
u'(0)=0,\quad
\cos \theta\,u(1)+\sin \theta\,u'(1)=0,
\end{gathered}
\end{equation}
{where $\theta\in [0,\pi/2]$ and
$\lambda^{2}$ is the spectral parameter}. Here
for the family $\{j_{\alpha}(q_{k}t)\}_{k=1}^{\infty},$ we assume the Dirichlet conditions $\theta=0$ and $u(1)=0$
 and, for $\{1,j_{\alpha}(q_{k}'t)\}_{k=1}^{\infty}$, the Neumann conditionds $\theta=\pi/2$ and $u'(1)=0$.


 In virtue of \eqref{j-diff-rec},
 we have 
\[
\cos \theta\,j_{\alpha}(\lambda)-
\sin \theta\,\frac{\lambda^{2}}{2(\alpha+1)}\,j_{\alpha+1}(\lambda)=0,
\]
or, equivalently,
\begin{equation}\label{sturm-liuv-bess+}
\cos \theta\,J_{\alpha}(\lambda)-\sin \theta\,\lambda J_{\alpha+1}(\lambda)=
AJ_{\alpha}(\lambda)+B\lambda J_{\alpha}'(\lambda)=0,
\end{equation}
where $A=\cos \theta-\alpha \sin \theta$, $B=\sin \theta$. 
{Since} $A/B+\alpha=\tan \theta\ge 0$, $\alpha>-1$,
{we have that} equation \eqref{sturm-liuv-bess+} has only real roots {(see \cite[Sec.~7.9]{BE53})}.
 Due to evenness, 
  it is enough to consider only nonnegative zeros, which we denote by
 $0\le r_{1}<r_{2}<\dots$. Then the eigenvalues and the eigenfunctions of the Sturm--Liouville problem
\eqref{sturm-liuv-bess} are $r_{k}^{2}$ and
$j_{\alpha}(r_{k}t)$, $k\in \mathbb{N}$, respectively. In particular, we have $r_{k}=q_{k}$ for
$\theta=0$ and $r_{k}=q_{k-1}'$ for $\theta=\pi/2$, where we put $q_{0}'=0$.

The Sturm--Liouville problem (\ref{sturm-liuv-bess}) is a particular case of the
problem (\ref{sturm-liuv-vsp}); take
 $K=G=w$, $L=0$, $r=\lambda^{2}$, $h=0$, and
$H=\cot \theta$. Then the statement of part~(i) is valid for the interval
$(0,1)$. In order to include the endpoints, we first prove part~(ii).



Let us show that the family $\{j_{\alpha+1}(q_kt)\}_{k=1}^{\infty}$ is the Chebyshev system on $(0,1)$.
Assume that the polynomial 
$P(t)=\sum_{k=1}^{n}A_{k}j_{\alpha+1}(q_{k}t)$ has $n$ zeros on $(0,1)$.
We consider $F(t)=t^{2\alpha+2}P(t)$. It {has at least} $n+1$ zeros including $t=0$.
By Rolle's theorem, for a smooth real function $f$ one has $N_{(a,b)}(f')\ge
N_{(a,b)}(f)-1$ (see \cite{BH17}).
 Thus, $P'$ has at least
$n$ zeros on $(0,1)$. In light of \eqref{j1-diff}, we obtain
\[
F'(t)=2(\alpha+1)t^{2\alpha+1}\sum_{k=1}^{n}A_{k}j_{\alpha}(q_{k}t).
\]
This contradicts the fact that
$\{j_{\alpha}(q_{k}t)\}_{k=1}^{\infty}$ is the Chebyshev system on $(0,1)$.

To prove that $\{j_{\alpha}(q_{k}'t)-j_{\alpha}(q_{k}')\}_{k=1}^{\infty}$ is the Chebyshev system on $(0,1)$, assume that
$P(t)=\sum_{k=1}^{n}A_{k}(j_{\alpha}(q_{k}'t)-j_{\alpha}(q_{k}'))$ has $n$
zeros on $(0,1)$. Taking into account the zero $t=1$,
 its derivative (see \eqref{j-diff-rec})
\[
P'(t)=-\frac{t}{2\alpha+2}\sum_{k=1}^{n}A_{k}q_{k}'^{2}j_{\alpha+1}(q_{k}'t)
\]
has at least $n$ zeros on $(0,1)$. This contradicts the fact that
 $\{j_{\alpha+1}(q_{\alpha+1,k}t)\}_{k=1}^{\infty}$ is the Chebyshev system on $(0,1)$.

Now we are in a position to show that the first system in \eqref{bess-sys} is Chebyshev on $[0,1)$.
Note that if 
$P(t)=\sum_{k=1}^{n}A_{k}j_{\alpha}(q_{k}t)$ has $n$ zeros on $[0,1)$, then always $P(0)=0$. Moreover, $P(1)=0$. Therefore, $P'$ has at least $n$ zeros on $(0,1)$, which is impossible since $P'(t)=
-\frac{t}{2\alpha+2}\sum_{k=1}^{n}A_{k}q_{k}^{2}j_{\alpha+1}(q_{k}t)$ and
$j_{\alpha+1}(q_{k}t)$ is the Chebyshev system on $(0,1)$.

Similarly, if $P(t)=\sum_{k=0}^{n-1}A_{k}j_{\alpha}(q_{k}'t)$
 (we assume $q_{0}'=0$) has $n$ zeros on $[0,1]$, then one of the endpoints is a zero. Then
$P'(t)=-\frac{t}{2\alpha+2}\sum_{k=1}^{n-1}A_{k}q_{k}'^{2}j_{\alpha+1}(q_{k}'t)$
has at least $n-1$ zeros in $(0,1)$, which is impossible for Chebyshev system
$\{j_{\alpha+1}(q_{\alpha+1,k}t)\}_{k=1}^{\infty}$.
\end{proof}

\section{
An alternative proof of
positive definiteness of the function $g_{\alpha,m}$}\label{sec-lim}

In \cite{CC18},
the positive definiteness of the function $g_{d/2-1,m}$ given by \eqref{def-gam} was proved based on the use of classical translation operator in $\mathbb{R}^d$. This causes the restriction $\alpha=d/2-1$.
Another approach to see that $g_{\alpha,m}$ is positive definite, is to employ Bochner's theorem and
show that the Hankel transform of $g_{\alpha,m}$ is nonnegative, which is equivalent to fact that
the matrix of the generalized translations $(T_{\alpha}^{x_i}f(x_j))_{i,j=1}^{N}$ is positive definite, see Section~\ref{sec-H}.
Here we follow this approach and ideas from \cite{CC18}.






Let
$
R_{n}^{(\alpha)}(\theta)=\frac{P_{n}^{(\alpha,\alpha)}(\theta)}{P_{n}^{(\alpha,\alpha)}(1)}
$
be the normalized Jacobi polynomial and $-1<r_n<\dots<r_1<1$ be its zeros, see, e.g., \cite{Se62}.
Define the generalized translation operator on $[-1,1]$
as follows\begin{equation}\label{eq50}
\tau^{\theta}f(\rho)=c_{\alpha}\int_{0}^{\pi}f(\sqrt{1-\theta^2}
\sqrt{1-\rho^2}+2\theta\rho\cos\varphi)\sin^{2\alpha}\varphi\,d\varphi,
\end{equation}
{where $c_{\alpha}$ is defined in} \eqref{eq7}.
We remark that
$
\tau^{\theta}R_n^{(\alpha)}(\rho)=R_n^{(\alpha)}(\theta)R_n^{(\alpha)}(\rho).
$

Consider the polynomial
$p_{n-k}(\theta)=\frac{R_n^{(\alpha)}(\theta)}{(\theta-r_1)\cdots(\theta-r_k)}.
$
It was shown in \cite{CK07} that
\[
p_{n-k}(\theta)=\sum_{s=0}^{n-k}a_sR_s^{(\alpha)}(\theta),\quad a_s\ge 0,\quad
i=0,\dots,n-k.
\]
This implies that for any choice of $\theta_1,\dots, \theta_N\subset [-1,1]$
the matrix $(\tau^{\theta_i}p_{n-k}(\theta_j))_{i,j=1}^n$ is { positive semidefinite}, i.e.,
\begin{align*}
&\sum_{i,j=1}^Nc_i\overline{c_j}\,\tau^{\theta_i}p_{n-k}(\theta_j)=\sum_{s=0}^{n-k}a_s
\sum_{i,j=1}^Nc_i\overline{c_j}\,\tau^{\theta_i}R_s^{(\alpha)}(\theta_j)\\
&\quad =\sum_{s=0}^{n-k}a_s\sum_{i,j=1}^Nc_i\overline{c_j}\,R_s^{(\alpha)}(\theta_i)R_s^{(\alpha)}(\theta_j)=
\sum_{s=0}^{n-k}a_s\Bigl|\sum_{i=1}^Nc_iR_s^{(\alpha)}(\theta_i)\Bigr|^2\ge 0.
\end{align*}

Recall again that $q_{i}=q_{\alpha,i}$ are zeros of $j_\alpha(y)$ and
$
g_k(y)=\frac{j_{\alpha}(y)}{(q_1^2-y^2)\cdots (q_k^2-y^2)}.
$ 
We note (see \cite[Sec. 8.1]{Se62}) that
\[
\lim\limits_{n\to \infty}R_n^{(\alpha)}\Bigl(1-\frac{y^2}{2n^2}+o\Bigl(\frac{1}{n^2}\Bigr)\Bigr)=j_{\alpha}(y)
\]
 uniformly in $y\in [0,L]$ for any positive $L$.
Since (\cite[Sec. 8.1]{Se62})
\[
r_i=1-\frac{q_i^2}{2n^2}+o\Bigl(\frac{1}{n^2}\Bigr),
\]
then setting $\theta=1-y^2/(2n^2)+o(1/n^2)$, we obtain
\[
\lim\limits_{n\to
\infty}(2n^2)^k(\theta-r_1)\cdots(\theta-r_k)=(q_1^2-y^2)\cdots(q_k^2-y^2)
\]
 uniformly in $y\in [0,L]$. 

Let us show that there holds 
\begin{equation}\label{eq52}
\lim\limits_{n\to
\infty}(2n^2)^{-k}p_{n-k}\Bigl(1-\frac{y^2}{2n^2}+o\Bigl(\frac{1}{n^2}\Bigr)\Bigr)=g_k(y)
\end{equation}
uniformly in $y\in [0,L]$.
This is true on any interval without arbitrarily small neighborhoods of points $q_i$, $i=1,\dots,k$.
 Without loss of generality, it is enough to consider
   a small neighborhood of $q_1$,
since $({q_2^2-y^2})\cdots({q_k^2-y^2})$ is bounded away from zero in this neighborhood.


Using \eqref{eq3} implies
\begin{align*}
\frac{j_{\alpha}(y)}{q_1^2-y^2}&=
\frac{j_{\alpha}(y)-j_{\alpha}(q_1)}{q_1^2-y^2}
=
\sum_{\nu=1}^{\infty}\frac{(-1)^{\nu}\Gamma(\alpha+1)}{4^{\nu}\nu!\,
\Gamma(\nu+\alpha+1)}\,\frac{y^{2\nu}-q_1^{2\nu}}{q_1^2-y^2}
\\
&=\sum_{\nu=1}^{\infty}
\frac{(-1)^{\nu-1}\Gamma(\alpha+1)}{4^{\nu}\nu!\,
\Gamma(\nu+\alpha+1)}\sum_{s=0}^{\nu-1}y^{2s}q_1^{2(\nu-1-s)}
\\
&=\sum_{s=0}^{\infty}\Bigl(\frac{y^2}{q_1^2}\Bigr)^s\sum_{k=s}^{\infty}\frac{(-1)^k\Gamma(\alpha+1)q_1^{2k}}{4^{k+1}(k+1)!\,
\Gamma(k+\alpha+2)}\\
&=\frac{1}{4}\sum_{s=0}^{\infty}\sum_{l=0}^{\infty}\frac{\Gamma(\alpha+1)}{\Gamma(s+l+2)\Gamma(s+l+\alpha+2)}
\Bigl(-\frac{y^2}{4}\Bigr)^s\Bigl(-\frac{q_1^2}{4}\Bigr)^l.
\end{align*}
Similarly, if $\theta=1-y^2/(2n^2)+o(1/n^2)$, then \cite[Sec. 4.21]{Se62}
\begin{align*}
&{n+\alpha\choose n}\frac{R_n^{(\alpha)}(\theta)}{\theta-r_1}=\sum_{\nu=1}^{n}
\frac{\Gamma(n+\nu+2\alpha+1)\Gamma(n+\alpha+1)((\theta-1)^{\nu}-(r_1-1)^{\nu})}
{2^{\nu}\nu!\,\Gamma(n-\nu+1)\Gamma(n+2\alpha+1)\Gamma(\nu+\alpha+1)(\theta-r_1)}\\
&\quad =\sum_{\nu=1}^{n}\frac{\Gamma(n+\nu+2\alpha+1)\Gamma(n+\alpha+1)}
{2^{\nu}\nu!\,\Gamma(n-\nu+1)\Gamma(n+2\alpha+1)\Gamma(\nu+\alpha+1)}\sum_{s=0}^{\nu-1}(\theta-1)^s(r_1-1)^{\nu-s-1}\\
&\quad=\sum_{s=0}^{n-1}\sum_{\nu=s}^{n-1}\frac{\Gamma(n+\nu+2\alpha+2)\Gamma(n+\alpha+1)(\theta-1)^s(r_1-1)^{\nu-s}}
{2^{\nu+1}(\nu+1)!\,\Gamma(n-\nu)\Gamma(n+2\alpha+1)\Gamma(\nu+\alpha+2)}\\
&\quad=\frac{1}{2}\sum_{s=0}^{n-1}\sum_{l=0}^{n-1-s}\frac{\Gamma(n+s+l+2\alpha+2)\Gamma(n+\alpha+1)(\theta-1)^s(r_1-1)^{l}}
{2^{s+l}\Gamma(s+l+2)\Gamma(n-s-l)\Gamma(n+2\alpha+1)\Gamma(s+l+\alpha+2)}\\
&\quad=\frac{1}{2}\sum_{s=0}^{n-1}\sum_{l=0}^{n-1-s}\frac{\Gamma(n+s+l+2\alpha+2)\Gamma(n+\alpha+1)(-y^2/4)^s(-q_1^2/4)^l(1+o(1/n^2))}
{n^{2(s+l)}\Gamma(s+l+2)\Gamma(n-s-l)\Gamma(n+2\alpha+1)\Gamma(s+l+\alpha+2)}.
\end{align*}
Since
\[
\frac{\Gamma(n+a)}{\Gamma(n+b)}\sim n^{a-b},\quad {n+\alpha \choose n}\sim
\frac{n^{\alpha}}{\Gamma(\alpha+1)},\quad n\to\infty,
\]
then, for fixed $s$ and $l$,
\begin{align}
&\frac{\Gamma(\alpha+1)\Gamma(n+s+l+2\alpha+2)\Gamma(n+\alpha+1)(1+o(1/n^2))(-y^2/4)^s(-q_1^2/4)^l}{4n^{2(s+l+1)+\alpha}
\Gamma(n-s-l)\Gamma(n+2\alpha+1)\Gamma(s+l+2)\Gamma(s+l+\alpha+2)} \nonumber \\
&\quad \sim
\frac{\Gamma(\alpha+1)(-y^2/4)^s(-q_1^2/4)^l}{4\Gamma(s+l+2)\Gamma(s+l+\alpha+2)},\quad
n\to\infty, \label{zv22}
\end{align}
and, for $\theta=1-y^2/(2n^2)+o(1/n^2)$, we have, uniformly on $y\in [0,L]$,
\[
\lim_{n\to\infty} 2^{-1}n^{\alpha-2}\,
\frac{R_n^{(\alpha)}(\theta)}{\theta-r_1}=\frac{j_{\alpha}(y)}{q_1^2-y^2}.
\]
We should explain how we take the limit under the sum. Since for any $n\ge 1$,
$0\le s\le n-1$, $0\le l \le n-1-s$,
\[
\frac{\Gamma(n+s+l+2\alpha+2)}{\Gamma(n+2\alpha+1)}\le (2n+2\alpha)^{s+l+1},
\]
\[
\frac{\Gamma(n)}{\Gamma(n-s-l)}\le n^{s+l},\quad
\frac{\Gamma(n+\alpha+1)}{\Gamma(n)}\le C(\alpha)n^{\alpha+1},
\]
then (\ref{zv22}) can be bounded from above by 
\begin{align*}
&\frac{\Gamma(\alpha+1)\Gamma(n+s+l+2\alpha+2)\Gamma(n+\alpha+1)
|1+o(1/n^2)|(y^2/4)^s(q_1^2/4)^l}{4n^{2(s+l+1)+\alpha}
\Gamma(n-s-l)\Gamma(n+2\alpha+1)\Gamma(s+l+2)\Gamma(s+l+\alpha+2)}\\ &\quad \le
C_1(\alpha)\,\frac{(y^2/4)^s(q_1^2/4)^l}{\Gamma(s+l+2)\Gamma(s+l+\alpha+2)}.
\end{align*}
Moreover, the following series converges uniformly on any interval $[0,L]$, $q_1\le L$,
\begin{align*}
&\sum_{s=0}^{\infty}\sum_{l=0}^{\infty}\frac{(y^2/4)^s(q_1^2/4)^l}{\Gamma(s+l+2)\Gamma(s+l+\alpha+2)}\\
&\quad \le \sum_{s=0}^{\infty}\sum_{l=0}^{\infty}\frac{(L^2/4)^{s+l}}{\Gamma(s+l+2)\Gamma(s+l+\alpha+2)}\le
\sum_{s=0}^{\infty}\sum_{l=0}^{\infty}\frac{(L^2/4)^{s+l}}{(s+l+1)!}\\
&\quad =\sum_{m=0}^{\infty}(m+1)\,
\frac{(L^2/4)^{m}}{(m+1)!}=\sum_{m=0}^{\infty}\frac{(L^2/4)^{m}}{m!}=e^{L^2/4}.
\end{align*}
Thus, (\ref{eq52}) is proved.

Let
\[
x_i\in [0,\infty),\quad \frac{x_i}{n}\le 1,\quad
\theta_i=\sqrt{1-\Bigl(\frac{x_i}{n}\Bigr)^2},\quad i=1,\dots,N.
\]
For $i,j=1,\dots,N$, there holds, uniformly on $\varphi\in [0,\pi]$ and
for sufficiently large $n$,
\[
\sqrt{1-\Bigl(\frac{x_i}{n}\Bigr)^2}\sqrt{1-\Bigl(\frac{x_j}{n}\Bigr)^2}+
2\,\frac{x_ix_j}{n^2}\cos\varphi=1
-\frac{y_{ij}^2}{2n^2}+o\Bigl(\frac{1}{n^2}\Bigr),
\]
 where
\[
y_{ij}=\sqrt{x_i^2+x_j^2-2x_ix_j\cos\varphi}.
\]
Therefore, by \eqref{eq52} and the definitions of
the
generalized translation operator
\eqref{eq10} and
\eqref{eq50}, for any $i,j$,
\[
\lim\limits_{n\to\infty}(2n^2)^{-k}\tau^{\theta_i}p_{n-k}(\theta_j)=
T_{\alpha}^{x_i}g_k(x_j).
\]
Since the matrix $(\tau^{\theta_i}p_{n-k}(\theta_j))$ is positive semidefinite, then the matrix $(T_{\alpha}^{x_i}g_k(x_j))$
is also positive semidefinite. Then, by Levitan's theorem, $\mathcal{H}_{\alpha}(g_k)(t)\ge 0$ and the functions
$g_k$ and \eqref{def-gam} are positive definite.

\vspace{5mm}
{\bf{Acknowledgements.}}
We would like to thank  P.~B\'erard who pointed out the possibility to apply Liouville's method to show Theorem~\ref{prop-bess} and
E.~Berdysheva for helpful comments.}

\end{document}